\documentclass[11pt]{article}

\usepackage[utf8]{inputenc}
\usepackage[T1]{fontenc}
\usepackage{amsmath}
\usepackage{amssymb}
\usepackage{amsthm}
\usepackage{amsfonts}
\usepackage{amssymb}
\usepackage[pdftex]{graphicx}

\newtheorem{theorem}{Theorem}
\newtheorem{remark}[theorem]{Remark}
\newtheorem{lemma}[theorem]{Lemma}

\newtheorem{corollary}[theorem]{Corollary}
\newtheorem{example}[theorem]{Example}

\def\setR{{\mathbb{R}}}
\def\setN{{\mathbb{N}}}
\newcommand{\sth}{\!:\,}
\newcommand{\nl}{\nonumber\\ }
\newcommand{\calI}{{\mathcal{I}}}
\newcommand{\calT}{{\mathcal{T}}}
\newcommand{\calU}{{\mathcal{U}}}
\newcommand{\calM}{{\mathcal{M}}}
\renewcommand{\phi}{\varphi}
\renewcommand{\epsilon}{\varepsilon}

\def\calB{{\cal B}}
\newcommand{\eref}[1]{{\rm(\mbox{\ref{#1}})}}
\newcommand{\IVP}[2]{IVP \eref{#1}--\eref{#2}}
\newcommand{\IIVP}[2]{IIVP \eref{#1}--\eref{#2}}
\newcommand{\cball}[2]{{\overline\calB_{#1}\!\left({#2}\right)}}

\newcommand{\PC}{{PC([-M_1,\alpha],\setR^n)}}
\newcommand{\PCr}{{PC([-r,\alpha],\setR^n)}}
\newcommand{\PCb}{{PC([0,\alpha],\setR^n)}}
\newcommand{\essinf}{\mathop{\rm ess\,inf}}
\newcommand{\esssup}{\mathop{\rm ess\,sup}}

\begin{document}

\title{On existence, uniqueness and numerical approximation of impulsive differential equations with
adaptive state-dependent delays using equations with piecewise-constant arguments}
\author{Ferenc Hartung\\
Department of Mathematics\\
University of Pannonia\\
8201 Veszpr\'{e}m, P.O. Box 158, Hungary}

\maketitle

\begin{abstract}
In this paper we consider a  class of impulsive nonlinear differential equations
with adaptive state-dependent delays. We discuss the existence and uniqueness of solutions  of the initial value
problem using a Picard--Lindel\"of type argument where we define approximate solutions with the
help of equations with piecewise-constant arguments (EPCAs).
Moreover, we show that the solutions of the associated EPCAs approximate the solutions of the original
impulsive DDE with adaptive state-dependent delay uniformly on compact time intervals.
The key assumption underlying both results is that the delayed time function is monotone, or  piecewise strictly monotone.
\end{abstract}

\textbf{AMS(MOS) subject classification:} 34K43, 34K45, 65L03

\textbf{keywords:} Delay differential equation,
state-dependent delay,
adaptive delay,
equations with piecewise-constant arguments,
numerical approximation.
\bigskip

\section{Introduction}

State-dependent delay equations (SD-DDEs) are frequently  used in  mathematical models (see \cite{HKWW2006} for a survey of applications and basic theory, and, e.g., \cite{GHMWW2025, MSBK2025, SWLLL2025} for some recent papers). Therefore the qualitative properties of SD-DDEs is still  an actively investigated research    \cite{BK2020, BGK2025, FW2024, HW2023, LG2023, LP2022, ML2024, W2023}.
In many applications the dependence of the delay on the state of the system is defined by an algebraic or an integral equation (see \cite{HKWW2006} for some applications). Another class of SD-DDEs is when the delay function is defined by an associated differential equation
which depends on the state of the model. Such definition of a delay function is called adaptive delay. See
\cite{AHH1998, CHW2010,  Ha2024, HKWW2006, MA2000,  PL2020, W2006, WZHM2011} for some applications and
qualitative investigations of SD-DDEs with adaptive delay.
In some models due to the sudden change in the  environment such as
weather, radiation, drug administration, harvesting, etc., impulsive delay differential equations are used (see, e.g., \cite{FO2020, LS2022, SA2009}). There are only a few papers which study SD-DDEs with impulses (see \cite{A2007, C2022, FHI2019, GTM2023, Ha2024, ML2024}).

In this paper we consider a class of SD-DDEs of the form
\begin{equation}\label{e401}
 \dot x(t) = f(t,x(t),x(t-\tau(t)),\theta),\qquad \mbox{a.e.}\ t\in [0,T],
\end{equation}
where the delay function is defined by the adaptive condition
\begin{equation}\label{e401b}
 \dot \tau(t) = g(t,x(t),\tau(t)),\qquad \mbox{a.e.}\ t\in [0,T],
\end{equation}
the impulsive conditions are
\begin{equation}\label{e401c}
 \Delta x(t_k) = I_k(x(t_k-)),\qquad k=1,2,\ldots,K,
\end{equation}
where the associated  initial condition to \eref{e401} is
\begin{equation}\label{e402a}
x(t) = \phi(t),\qquad t\in(-\infty,0],
\end{equation}
and the initial condition associated to \eref{e401b} is
\begin{equation}\label{e402}
  \tau(0)=\lambda.
\end{equation}
Here  the initial delay value $\lambda$ is a positive constant,
$0<t_1<t_2<\cdots<t_K<T$ are the impulsive moments, and $\Delta x(t):=x(t)-x(t-)$.
We also assume that the initial function is piecewise continuous: on any finite interval the function $\phi$ can have only finitely many jump discontinuities, and at all such points it is right-continuous  (see assumption (A4) below for details).

By a solution of the impulsive initial value problem (IIVP) \eref{e401}--\eref{e402} on the interval $(-\infty,\alpha]$ we mean a
function $x\sth(-\infty,\alpha]\to\setR^n$ and $\tau\sth[0,\alpha]\to\setR$, where $x$ is absolutely continuous on the intervals $[t_k,t_{k+1})\subset[0,\alpha]$ for all $k$,  $x$ is right-continuous at the impulsive moments $t_k$, and
$\tau$ is continuous on $[0,\alpha]$, continuously differentiable on $[0,\alpha]$ except at the points
$\{t_1,\ldots,t_K\}\cap[0,\alpha]$, $x$ is right-continuous at the impulsive moments $t_k$,
 and  $x$ and $\tau$ satisfy \eref{e401}--\eref{e402}.

The main goal of this manuscript is to investigate local existence and uniqueness of the solutions of the \IIVP{e401}{e402},
and also the numerical approximation of the solutions.
The recent paper \cite{C2022} examines, among others, the existence and uniqueness of the solutions of impulsive SD-DDEs with explicitly defined delay function. The proof of the existence uses the Schauder Fixed Point Theorem and the assumption that the delayed time function is monotone increasing.
We note that in \cite{Ha2024} a variant of  the \IIVP{e401}{e402} was studied where the formula of $f$ and $g$ depend on
additional parameters, possibly on infinite dimensional parameters. The main goal of that paper was to study differentiability
of the solution with respect to (wrt) those parameters. Therefore it was essential in \cite{Ha2024} to assume continuous differentiability of the functions $f$ and $g$.
Moreover, the boundedness of the partial derivatives of $f$ and $g$ wrt
their variables (except the time variable) was assumed, which clearly imply global Lipschitz continuity of $f$ and $g$.
This class of conditions was practical in the presence of the infinite dimensional parameters. But if we check the proof
of the well-posedness result, Theorem 3.3 in \cite{Ha2024}, in fact, the continuous differentiability of $f$ and $g$ was not used,
only Lipschitz continuity of $f$ and $g$ was needed for the proof of well-posedness. In this paper, for the class of
the \IIVP{e401}{e402}, we need only local Lipschitz continuity of the respective functions,
which clearly allows larger classes of examples for the applications.
Note that the assumption that the initial function $\phi$
is Lipschitz continuous is also  standard  for the uniqueness of the solutions for  classes of SD-DDEs without impulses, see e.g.  \cite{Dr1961, Ha1997}.
We assume piecewise Lipschitz continuity of the initial function
 for the well-posedness and also for the convergence of the numerical method
(see \cite{GyHT1995, Ha1997, HHT1997} for similar results for other classes of SD-DDEs without impulses).

In \cite{Ha2024} it was assumed that $g(t,u,w)\leq q<1$ holds for all
$t\in[0,T]$, $u\in\setR^n$ and $w\in\setR$, which guarantees that the delayed time function $t-\tau(t)$ is strictly monotone
increasing. This property was essential in \cite{Ha2024} in the proofs of differentiability wrt parameters. See also \cite{AHH1998, MA2000} where similar condition was assumed.  In Theorem~\ref{t1} below, we assume this condition and establish existence and uniqueness when the delayed time function is monotone increasing. For convergence of our numerical approximation scheme (see Theorem~\ref{t2} below), we require only the weaker hypothesis that the delayed time function is piecewise strictly monotone (see Section~2 for the definition). We also present an existence and uniqueness result under this piecewise strict monotonicity assumption (see Theorem~\ref{t3}). But then we do not have any a priori lower estimate of the delayed time function.
Therefore, we cannot assume that the initial interval has a predefined finite length. So we associate initial condition \eref{e402a} to \eref{e401} on
an infinite interval $(-\infty,0]$. Of course, if a solution is defined on a finite interval, then the delayed time function
is bounded below by a negative finite constant along the solution.

In \cite{Ha2024} the Schauder Fixed Point Theorem combined with the method of steps
was used to prove existence of the solution. In this manuscript we use the classical Picard--Lindel\"of type method
to show the existence of the IIVP.
We define a numerical approximation scheme with the help of equations with piecewise-constant arguments (EPCAs).
We take a sequence of approximate solutions, and show that a subsequence converges to a solution of our IIVP.
A similar argument was used in \cite{HHT1997} for a class of neutral SD-DDEs.
EPCAs were first used in \cite{Gy1991} to define a numerical approximation scheme
and to prove its convergence for a
class of linear delay and neutral equations with constant delays. Later, similar schemes were defined and studied
in different classes of SD-DDEs \cite{GyH2008,GyHT1995,Ha2022,HHT1997}.
The investigation of EPCAs goes back to the works of Cooke and Wiener \cite{CB1982,Wi1983}, but it is still an
active research area, see, e.g., \cite{Ak2011, CPT2019, LOX2025}.

This paper is organized as follows. Section~\ref{s2} introduces some notations and
preliminary results, Section~\ref{s3} defines our numerical scheme and discusses the local existence and uniqueness results related to  the \IIVP{e401}{e402}.
Section~\ref{s4} contains the proof for the convergence of the numerical scheme, and Section~\ref{s5} presents two
numerical examples to illustrate the results of this manuscript.

\bigskip

%%%%%%%%%%%%%%%%%%%%%%%%%%%%%%%%%%%%%%%%%%%%%%%%%%%%%%%%%%%%%%%%%%%%%%%%%%%%%%%%%%%%%%%%%%%%%%
%%%%%%%%%%%%%%%%%%%%%%%%%%%%%%%%%%%%%%%%%%%%%%%%%%%%%%%%%%%%%%%%%%%%%%%%%%%%%%%%%%%%%%%%%%%%%%
%%%%%%%%%%%%%%%%%%%%%%%%%%%%%%%%%%%%%%%%%%%%%%%%%%%%%%%%%%%%%%%%%%%%%%%%%%%%%%%%%%%%%%%%%%%%%%

\section{Notations and preliminaries}
\label{s2}

We use $\setN_0$ and $\setR$ to denote the sets of nonnegative integers and reals, respectively.
For $h>0$ we use the notation $\setN_0h=\{jh\sth j\in\setN_0\}$.
A fixed norm on $\setR^n$ is denoted by $|\cdot|$. We will use the notation
$\cball{\setR^n}{M}=\{u\in\setR^n\sth |u|\leq M\}$ for the closed ball in $\setR^n$ with
radius $M$ centered at the origin.

\bigskip

Let $r>0$ be fixed, and
consider a fixed finite sequence $-r<t_{-\ell_0}<t_{-\ell_0+1}<\cdots<t_0=0<t_1<\cdots< t_{k_0}<\alpha$. $\PCr$ denotes the space of piecewise continuous functions
$x\sth [-r,\alpha]\to\setR^n$ with discontinuity points at $\{t_{-\ell_0},\ldots,t_{k_0}\}$, where the left-hand limits $x(t_k-)$ exist, and
it is right-continuous, i.e., $x(t_k)=x(t_k+)$ for $k=1,\ldots,k_0$.
Note that in the notation $\PCr$ the dependence on the set $\{t_{-\ell_0},\ldots,t_{k_0}\}$ of the fixed
time discontinuity points  is omitted for simplicity, but it always should be kept in mind.
We have that $\PCr$ is a Banach space with the norm
$|x|_{\PCr}=\sup\{ |x(t)|\sth t\in[-r,\alpha]\}$.

Associated to the discontinuity points  we use the notations
$$
t_{-\ell_0-1}=-r,\quad t_{{k_0}+1}=\alpha,\quad \calI_k'=[t_k,t_{k+1})\ \mbox{for}\ k={-\ell_0-1},\ldots,{k_0}-1,
\quad \mbox{and}\quad \calI_{k_0}'=[t_{k_0},\alpha].
$$

It is easy to check the following generalization of the Arzel\`a--Ascoli Theorem (see, e.g., \cite{SL2006}).

\begin{lemma}\label{l10}
 Let $U\subset \PCr$. Then $U$ is relatively compact if and only if
 \begin{itemize}
  \item[(i)] $U$ is uniformly bounded, i.e., there exists a constant $R>0$ such that
  $|x|_\PCr \leq R$ for $x\in U$,
  \item[(ii)] $U$ is quasiequicontinuous on $[-r,\alpha]$, i.e., for every $\epsilon>0$ there exists $\delta>0$
  such that $|x(t)-x(\bar t)|\leq \epsilon$ for $x\in U$ and $t,\bar t\in\calI_k'$ for $k=-\ell_0-1,\ldots,{k_0}$ and $|t-\bar t|\leq\delta$.
 \end{itemize}
\end{lemma}

Our proofs   below  will rely on the following estimate, which
generalizes a delayed version of the Gronwall's lemma for an impulsive integral inequality.
\medskip

\begin{lemma}[\cite{Ha2024}, Lemma 2.3]\label{l00}
Let  $a_0,a_1,a_2, b, c\geq0$,  $0<t_1<\cdots<t_K<T$, $r>0$ be fixed,
and let $t_0=0$ and $t_{K+1}=T$.
Suppose a function  $u\sth[-r,T]\to \setR^n$ is continuous on the intervals $[t_k,t_{k+1})$ for  $k=0,\ldots,K-1$ and
on $[-r,0]$ and $[t_K,T]$, has finite left-limits $u(t_k-)$ at the points $t_k$ for $k=1,\ldots,K$,
and satisfies
\begin{align*}
|u(t)|&\leq a_0,\qquad t\in[-r,0],\\
|u(t)|&\leq |u(t_k)|+a_1 +b\int_{t_k}^t \sup_{-r\leq \zeta\leq s}|u(\zeta)|\,ds,\qquad t\in[t_k,t_{k+1}),\quad k=0,\ldots,K,
\end{align*}
and
$$
|\Delta u(t_k)|\leq c|u(t_k-)|+a_2,\qquad k=1,\ldots,K.
$$
Then
\begin{equation}\label{e303}
|u(t)|\leq  \sum_{j=0}^K(1+c)^j (a_0+a_1+a_2)e^{bt},\quad t\in[0,T].
\end{equation}
\end{lemma}

\bigskip

An  absolutely continuous function $u\sth[a,b]\to\setR$ is called
 \emph{piecewise strictly monotone} on $[a,b]$ if  there exists a finite mesh
$a=s_0<s_1<\cdots<s_{m-1}<s_m=b$ of $[a,b]$ such that for all $i=0,1,\ldots,m-1$
 either
$$
\essinf \{ \dot u(s)\sth s\in[a',b']\}>0,\qquad \mbox{for all }\ [a',b']\subset(s_i,s_{i+1})
$$
or
$$
\esssup \{ \dot u(s)\sth s\in[a',b']\}<0,\qquad \mbox{for all }\ [a',b']\subset(s_i,s_{i+1}).
$$
This property was essential in \cite{Ha2013b} to prove differentiability of the solutions of a SD-DDE with
piecewise strictly monotone delayed time function. An example was presented that in the lack of this property
the differentiability result may fail.

%%%%%%%%%%%%%%%%%%%%%%%%%%%%%%%%%%%%%%%%%%%%%%%%%%%%%%%%%%%%%%%%%%%%%%%%%%%%%%%%%%%%%%%%%%%%%%
%%%%%%%%%%%%%%%%%%%%%%%%%%%%%%%%%%%%%%%%%%%%%%%%%%%%%%%%%%%%%%%%%%%%%%%%%%%%%%%%%%%%%%%%%%%%%%
%%%%%%%%%%%%%%%%%%%%%%%%%%%%%%%%%%%%%%%%%%%%%%%%%%%%%%%%%%%%%%%%%%%%%%%%%%%%%%%%%%%%%%%%%%%%%%

\section{Existence and uniqueness of solutions}\label{s3}
\setcounter{equation}0
\setcounter{theorem}0

Let $T>0$ be a finite constant.
For given fixed impulsive time moments $0<t_1<t_2<\cdots<t_K<T$, let $\calT:=\{ t_1,\ldots,t_K\}$
denote the set of the impulsive time moments.

Consider the nonlinear SD-DDE
\begin{equation}\label{e1}
\dot x(t) = f(t,x(t),x(t-\tau(t))),\qquad \mbox{a.e.}\ t\in [0,T],
\end{equation}
where the delay function is defined by the adaptive condition
\begin{equation}\label{e1b}
 \dot \tau(t) = g(t,x(t),\tau(t)),\qquad \mbox{a.e.}\ t\in [0,T],
\end{equation}
the impulsive conditions are
\begin{equation}\label{e1c}
 \Delta x(t_k) = I_k(x(t_k-)),\qquad k=1,2,\ldots,K,
\end{equation}
and the  initial conditions are
\begin{equation}\label{e2a}
x(t) = \phi(t),\qquad t\in(-\infty,0],
\end{equation}
and $\tau(0)=\lambda$, where $\lambda>0$ is a constant.
For simplicity of the  notations later, we extend $\tau$ to $(-\infty,0)$ by a constant value, and consider the extended
initial condition
\begin{equation}\label{e2}
\tau(t) = \lambda,\qquad t\in(-\infty,0].
\end{equation}

Next we list our assumptions.
\begin{itemize}
\item[(A1)]  
\begin{itemize}
\item[(i)] $f\, :\, \setR\times \setR^n\times\setR^n\supset[0,T]\times\setR^n\times\setR^n\to\setR^n$ is continuous,
\item[(ii)] $f$ is locally Lipschitz continuous wrt its second and third arguments, i.e.,
    for every $M>0$ there exists $L_1=L_1(M)\geq0$ such that
    $$
   |f(t,u,v)-f(t,\bar u,\bar v)|\leq L_1\Bigl(|u-\bar u|+|v-\bar v|\Bigr),\qquad t\in[0,T],\ u,\bar u,v,\bar v\in\cball{\setR^n}{M}.
    $$
\end{itemize}
\item[(A2)]
\begin{itemize}
\item[(i)] $g\, :\, \setR\times \setR^n\times\setR\supset[0,T]\times\setR^n\times\setR	\to\setR$ is continuous,
\item[(ii)] the function $g$ satisfies
$$
g(t,u,0)>0,\qquad t\in [0,T],\quad u\in \setR^n,
$$
\item[(iii)] $g$ is locally Lipschitz continuous wrt its second and third arguments, i.e., for every $M>0$
    there exists a constant $L_2=L_2(M)\geq0$ such that
    $$
     |g(t,u,w)-g(t,\bar u,\bar w)|\leq L_2\Bigl(|u-\bar u|+|w-\bar w|\Bigr),\qquad
     t\in[0,T],\ u,\bar u\in\cball{\setR^n}{M},\ w,\bar w\in[-M,M].
    $$
\end{itemize}
\item[(A3)]
The functions  $I_k\sth\setR^n\to\setR^n$ are globally Lipschitz continuous for $k=1,\ldots,K$, i.e.,
there exists a constant $L_3\geq 0$ such that
 $$
 |I_k(u)-I_k(\bar u)|\leq L_3|u-\bar u|,\qquad u,\bar u\in\setR^n,\quad k=1,2,\ldots,K.
 $$
\item[(A4)]
\begin{itemize}
\item[(i)] The initial function  $\phi\sth(-\infty,0]\to\setR^n$ is locally piecewise Lipschitz continuous, i.e.,
there exists a sequence $t_{-k}$ ($k\in\setN_0$), where $t_0=0$, and the strictly monotone decreasing sequence $t_{-k}$ tends to $-\infty$ as $k\to\infty$, and $\phi$ can have jump discontinuities only at the points $t_{-k}$, $\phi$ is right-continuous at $t_{-k}$ for $k\in\setN_0$, and
for every $M>0$ there exists a constant $L_4=L_4(M)\geq 0$ such that
 $$
 |\phi(s)-\phi(\bar s)|\leq L_4|s-\bar s|,\qquad s,\bar s\in [t_{-k-1},t_{-k})\cap[-M,0],\quad k\in\setN_0.
 $$
\item[(ii)]  $\phi$ is bounded on $(-\infty,0]$, i.e., there exists $N_\phi\geq0$ such that
 $$
 |\phi(s)|\leq N_\phi,\qquad s\in(-\infty,0].
 $$
 \end{itemize}
 \end{itemize}
\bigskip

Consider the jump discontinuity points of the initial function $t_{-k}$ ($k\in\setN_0$), and define $\calT^*=\calT\cup\{t_{-k}\sth k\in\setN_0\}$.
To simplify  the notation, we also introduce  $t_{K+1}:=T$,
and we define the intervals
$$
\calI_k:=[t_k,t_{k+1})\qquad \mbox{for}\ k=\cdots,-1,0,1,\ldots,K-1,\qquad \mbox{and}\quad \calI_K:=[t_k,T].
$$
For a fixed stepsize $h>0$ we introduce the notation
$$
[t]_h = \left[\frac th\right]\!h,
$$
where  $[\cdot]$ denotes the greatest integer part function. Then $t-h<[t]_h\leq t$, and hence
\begin{equation}\label{e216}
|[t]_h-t|\leq h,\qquad t\in\setR.
\end{equation}
The mesh points of our numerical approximation will be the points of the set $\setN_0h$.

Motivated by the numerical approximation scheme with the help of EPCAs
introduced by I.~Gy{\H o}ri in \cite{Gy1991}, and later studied in \cite{GyHT1995,Ha2022,HHT1997} for different classes of FDEs,
we consider the approximate system of impulsive delayed EPCA with an  adaptive delay  defined by
\begin{align}
 \dot x_h(t) &= f([t]_h,x_h([t]_h),x_h([t]_h-[\tau_h([t]_h)]_h)),\qquad t\in[0,T],\label{ae1}\\
 \dot \tau_h(t) &= g([t]_h,x_h([t]_h),\tau_h([t]_h)),\qquad t\in [0,T],\label{ae1b}\\
 \Delta x_h(t_k) &= I_k(x_h(t_k-)),\qquad k=1,2,\ldots,K,\label{ae1c}\\
 x_h(t) &= \phi(t),\qquad t\in(-\infty,0],\label{ae2a}\\
 \tau_h(t) &= \lambda,\qquad t\in(-\infty,0].\label{ae2}
\end{align}

The function $x_h$ has jump discontinuities at $t_k\in\calT$, it is right-continuous at $t_k$, and $\tau_h$ is
a continuous function. At a point $t\geq0$ where  $[t]_h\in\calT$ or $[t]_h-[\tau_h([t]_h)]_h\in\calT^*$ or for $t\in\setN_0h$, we interpret the derivatives in \eref{ae1} and \eref{ae1b} as right-derivatives.
\smallskip

We define the positive constant
\begin{equation}\label{e261}
h_0=\min\{t_{k+1}-t_k\sth k=0,\ldots,K\}.
\end{equation}
If $0<h<h_0$, then  between two consecutive mesh points $jh$ and $(j+1)h$, there is at most one impulsive time moment $t_k$, so
in this paper we always assume that $0<h<h_0$.
First consider the following remark about the computation and the existence of
solutions of \eref{ae1}--\eref{ae2}.

\begin{remark}\label{rem1}\rm
Suppose $t_k\leq jh<(j+1)h\leq t_{k+1}$. Then
integrating equations \eref{ae1} and \eref{ae1b} from $jh$ to $t\in(jh,(j+1)h)$, and taking the limit $t\to (j+1)h-$ we get
\begin{align*}
 x_h((j+1)h-)&=x_h(jh)+hf(jh,x_h(jh),x_h(jh-[\tau_h(jh)]_h)),\\
 \tau_h((j+1)h)&=\tau_h(jh)+hg(jh,x_h(jh),\tau_h(jh)).
\end{align*}
If $jh<t_{k}<(j+1)h$, then similarly to the previous calculation, and using the impulsive condition \eref{ae1c}, we
obtain
\begin{align*}
 x_h(t_k-)&=x_h(jh)+(t_{k}-jh)f(jh,x_h(jh),x_h(jh-[\tau_h(jh)]_h)),\\
 x_h(t_k) &= x_h(t_k-)+I_k(x_h(t_k-)),\\
 x_h((j+1)h-)&=x_h(t_k)+((j+1)h-t_{k})f(jh,x_h(jh),x_h(jh-[\tau_h(jh)]_h)),\\
 \tau_h((j+1)h)&=\tau_h(jh)+hg(jh,x_h(jh),\tau_h(jh)).
\end{align*}
Therefore, if $\tau_h(jh)\geq 0$ for  $j=0,1,\ldots,j_0$, then the \IIVP{ae1}{ae2} has a unique solution on the interval $(-\infty,(j_0+1)h]$.
If $\tau(jh)<0$ for some $j\in\setN$, then the relation defining $x_h((j+1)j-)$ is no longer an explicit recursion, so the existence
of $x_h((j+1)h-)$ is not obvious.
\end{remark}

By a solution of the \IIVP{ae1}{ae2} on $(-\infty,\alpha]$ for some $\alpha\leq T$  we mean a pair of functions
$(x_h,\tau_h)$, where
the function $\tau_h$ is continuous on
$[0,\alpha]$, and it is linear between consecutive mesh points $(\setN_0h\cap[0,\alpha])\cup\{\alpha\}$;
the function $x_h$ is continuous on the intervals
$\calI_k\cap[0,\alpha]$ for $k=0,1,\ldots,K$, and  it is linear between consecutive mesh points and impulsive time
moments $(\setN_0h\cap[0,\alpha])\cup\{\alpha\}\cup(\calT\cap[0,\alpha])$; $x_h$ has finite
left-sided limit, and it is right-continuous at each impulsive time moments $t_k\in\calT\cap[0,\alpha]$; and
\eref{ae1}--\eref{ae2} are satisfied.
\medskip

The following result will be essential in the proof of the existence and uniqueness theorem.

\begin{lemma}\label{l9}
 Assume (A1)--(A4), and let $\lambda>0$. Then there exist positive constants $h^*\leq h_0$ and $\alpha\leq T$
 such that $\alpha\not\in\calT$, and for $0<h<h^*$
the \IIVP{ae1}{ae2} has a unique solution $(x_h(t),\tau_h(t))$ on $[0,\alpha]$.
 Moreover,   there exist nonnegative constants $M_1,M_2,M_3$
such that  for $0<h<h^*$
\begin{align}
\tau_h(t)&>0,\qquad t\in[0,\alpha],\label{e414a}\\
|x_h(t)|+\tau_h(t)&< M_1,\qquad t\in(-\infty,\alpha],\label{e414}\\
|x_h(t)-x_h(\bar t)|&\leq M_2|t-\bar t|,\qquad  t,\bar t\in\calI_k\cap[0,\alpha],\ \, k=0,\ldots,K,\quad t,\bar t\in\calI_{-k}\cap[-M_1,0],\ \, k\in\setN_0\label{e415}\\
|\tau_h(t)-\tau_h(\bar t)|&\leq M_3|t-\bar t|,\qquad  t,\bar t\in[0,\alpha],\label{e415b}
\end{align}
\end{lemma}
\begin{proof}
Let $L_3$ be the Lipschitz constant from (A3), $N_\phi>0$ be the constant from (A4) (ii), and
fix a constant $M_1$ so that
$$
M_1>\Bigl(N_\phi+\lambda+T\max_{s\in[0,T]}|f(s,0,0)|+T\max_{s\in[0,T]}|g(s,0,0)|+\max_{k=1,\ldots,K}|I_k(0)|\Bigr)
\sum_{j=0}^K(1+L_3)^j.
$$
Consider the Lipschitz constants $L_1=L_1(M_1)$ and $L_2=L_2(M_1)$ from (A1) (ii) and (A2) (iii), respectively, and define
$$
L=L_1+L_2.
$$
Fix $0<\alpha\leq T$ such that $\alpha\not\in\calT$, and
\begin{equation}\label{e231}
\Bigl(N_\phi+\lambda+T\max_{s\in[0,T]}|f(s,0,0)|+T\max_{s\in[0,T]}|g(s,0,0)|+\max_{k=1,\ldots,K}|I_k(0)|\Bigr)
\sum_{j=0}^K(1+L_3)^j e^{2L\alpha}< M_1.
\end{equation}
It follows from Remark~\ref{rem1} that if $\tau_h$ takes a nonnegative value at a mesh point, then
$x_h$ and $\tau_h$ is uniquely defined at the next mesh point, so the solution can be extended to a longer interval.
Therefore, to prove that the   \IIVP{ae1}{ae2} has a unique solution on $[0,\alpha]$ for some $\alpha>0$,
it is enough to show that it generates a positive function $\tau_h$ on $[0,\alpha]$.

Since we assumed that $\tau_h(0)=\lambda>0$, $\tau_h(t)$ is positive for small $t$. Suppose there exists  $0<\alpha^*_h\leq \alpha$
such that $(x_h,\tau_h)$ exists on $[0,\alpha^*_h]$, and
\begin{equation}\label{e232}
\tau_h(t)>0,\qquad t\in[0,\alpha^*_h),\qquad\mbox{and}\qquad \tau_h(\alpha^*_h)=0.
\end{equation}
Note that the definitions of $N_\phi$ and $M_1$ yield  $|x_h(0)|+|\tau_h(0)|\leq N_\phi+\lambda<M_1$, so for small $t$ it follows
$|x_h(t)|+|\tau_h(t)|<M_1$.
We claim that
\begin{equation}\label{e888}
 |x_h(t)|+|\tau_h(t)|<M_1,\qquad t\in[0,\alpha^*_h].
\end{equation}
Suppose  that there exists $0<T^*_h\leq \alpha^*_h$ such that
\begin{equation}\label{e233}
|x_h(t)|+|\tau_h(t)|<M_1,\qquad t\in[0,T^*_h),\qquad\mbox{and}\qquad |x_h(T^*_h)|+|\tau_h(T^*_h)|\geq M_1.
\end{equation}

Since $x_h$ and $\tau_h$ are continuous on $\calI_k$, \eref{ae1} and \eref{ae1b} yield for $t\in\calI_k\cap [0,T^*_h]$
and $k=0,1,\ldots,K$
\begin{align}
 x_h(t) &= x_h(t_k)+\int_{t_k}^t f([s]_h,x_h([s]_h),x_h([s]_h-[\tau_h([s]_h)]_h))\,ds,\label{e281}\\
 \tau_h(t) &= \tau_h(t_k)+\int_{t_k}^t g([s]_h,x_h([s]_h),\tau_h([s]_h))\,ds.\label{e282}
\end{align}
Since \eref{e233} holds, we can use (A1) (ii) in the following estimate
for $t\in\calI_k\cap [0,T^*_h]$ and $k=0,1,\ldots,K$
\begin{align*}
 |x_h(t)| &\leq |x_h(t_k)|+\int_{t_k}^t |f([s]_h,0,0)|\,ds\\
 &\quad+ \int_{t_k}^t \Bigl|f([s]_h,x_h([s]_h),x_h([s]_h-[\tau_h([s]_h)]_h))-f([s]_h,0,0)\Bigr|\,ds\\
&\leq |x_h(t_k)|+T\max_{s\in[0,T]}|f(s,0,0)|
+ \int_{t_k}^t L_1\Bigl(|x_h([s]_h)|+|x_h([s]_h-[\tau_h([s]_h)]_h)|\Bigr)\,ds.
 \end{align*}
 Similarly, we have for $t\in\calI_k\cap [0,T^*_h]$ and $k=0,1,\ldots,K$
\begin{align*}
 |\tau_h(t)| &\leq |\tau_h(t_k)|+\int_{t_k}^t |g([s]_h,0,0)|\,ds
 + \int_{t_k}^t \Bigl|g([s]_h,x_h([s]_h),\tau_h([s]_h))-g([s]_h,0,0)\Bigr|\,ds\\
&\leq |\tau_h(t_k)|+T\max_{s\in[0,T]}|g(s,0,0)|
+ \int_{t_k}^t L_2\Bigl(|x_h([s]_h)|+|\tau_h([s]_h)|\Bigr)\,ds.
 \end{align*}
 Adding the two estimates and introducing the notation
 \begin{align*}
 \omega_h(t) &= |x_h(t)|+|\tau_h(t)|,\qquad t\in(-\infty,\alpha_h^*]
 \end{align*}
 we get for $t\in\calI_k\cap [0,T^*_h]$ and $k=0,1,\ldots,K$
$$
\omega_h(t)
\leq \omega_h(t_k)+T\Bigl(\max_{s\in[0,T]}|f(s,0,0)|+\max_{s\in[0,T]}|g(s,0,0)|\Bigr)
+ \int_{t_k}^t L\Bigl(\omega_h([s]_h)+\omega_h([s]_h-[\tau_h([s]_h)]_h)\Bigr)\,ds.
$$
 It follows  from \eref{e233} that
 $$
 [s]_h-[\tau_h([s]_h)]_h\geq -\tau_h([s]_h)\geq -M_1,\qquad s\in \calI_k\cap [0,T^*_h],\quad k=0,\ldots,K,
 $$
 hence for $t\in\calI_k\cap [0,T^*_h]$ and $k=0,1,\ldots,K$
\begin{align}
\omega_h(t)
&\leq \omega_h(t_k)+T\Bigl(\max_{s\in[0,T]}|f(s,0,0)|+\max_{s\in[0,T]}|g(s,0,0)|\Bigr)
+ \int_{t_k}^t 2L\max_{-M_1\leq\zeta\leq s}\omega_h(\zeta)\,ds.\label{e213a}
 \end{align}
The continuity of $\tau_h$, the impulsive condition \eref{ae1c} and assumption (A3) yield
for $t_k\leq T^*_h$
\begin{align*}
\Delta\omega_h(t_k)&=  \Bigl| |x_h(t_k)|-|x_h(t_k-)|\Bigr|\\
&\leq |\Delta x_h(t_k)|\\
&\leq|I_k(x_h(t_k-))-I_k(0)|+|I_k(0)|\\
&\leq L_3|x_h(t_k-)|+|I_k(0)|\\
&\leq L_3\omega_h(t_k-)+|I_k(0)|.
\end{align*}
The initial conditions \eref{ae2a} and \eref{ae2} and (A4) give
$$
\omega_h(t)=|x_h(t)|+|\tau_h(t)|=|\phi(t)|+\lambda\leq N_\phi+\lambda,\qquad t\in[-M_1,0].
$$
We apply  Lemma~\ref{l00} with $r=M_1$, $a_0=N_\phi+\lambda$, $a_1=T\Bigl(\max_{s\in[0,T]}|f(s,0,0)|+\max_{s\in[0,T]}|g(s,0,0)|\Bigr)$, $a_2=\max_{k=1,\ldots,K}|I_k(0)|$, $b=2L$
and $c=L_3$ to estimate \eref{e213a}, hence we get
\begin{equation}\label{e208a}
 \omega_h(t)\leq \Bigl(N_\phi+\lambda+T\max_{t\in[0,T]}|f(t,0,0)|+T\max_{t\in[0,T]}|g(t,0,0)|+\max_{k=1,\ldots,K}|I_k(0)|\Bigr)
\sum_{j=0}^K(1+L_3)^j e^{2L T^*_h},\quad t\in(-\infty,T^*_h].
\end{equation}
For $t=T^*_h\leq\alpha^*_h\leq\alpha$  relations \eref{e231} and \eref{e208a} contradict to \eref{e233}, hence such $T^*_h$ cannot exist, i.e.,
relation \eref{e888} holds.

Note that $t-\tau_h(t)\geq -M_1$ for $t\in[0,\alpha]$.
Let $L_4=L_4(M_1)$ be the Lipschitz constant from (A4) (i).
Define the constants
\begin{align}
 M_2&=\max\Bigl\{ \max\{|f(t,u,v)|\sth t\in[0,T], u,v\in \cball{\setR^n}{M_1}\},L_4\Bigr\},
 \label{e856}\\
 M_3&= \max\{|g(t,u,w)|\sth t\in[0,T], u\in \cball{\setR^n}{M_1}, w\in[0,M_1]\}.\label{e857}
\end{align}
Then it follows
\begin{align}
 |x_h(t)-x_h(\bar t)|
 &=\Bigl|\int_{\bar t}^t f([s]_h,x_h([s]_h),x_h([s]_h-[\tau_h([s]_h)]_h))\,ds\Bigr|\nl
 &\leq M_2|t-\bar t|,
 \qquad t,\bar t\in\calI_k\cap[0,\alpha^*_h],\quad k=0,\ldots,K,\nonumber\\%\label{e234}\\
 |x_h(t)-x_h(\bar t)|
 &=|\phi(t)-\phi(\bar t)| \leq L_4|t-\bar t|\leq M_2|t-\bar t|, \qquad t,\bar t\in\calI_{-k} \cap [-M_1,0],\quad k\in\setN_0,\nonumber\\
 |\tau_h(t)-\tau_h(\bar t)|
 &=\Bigl|\int_{\bar t}^t g([s]_h,x_h([s]_h),\tau_h([s]_h))\,ds\Bigr|\leq M_3|t-\bar t|,
 \qquad t,\bar t\in[0,\alpha^*_h].\label{e235}
\end{align}

In view of (A2) (i) and (ii), we obtain that
$$
A_0=\min\{g(t,u,0)\sth t\in[0,T],\ u\in\cball{\setR^n}{M_1}\}>0.
$$
The uniform continuity of $g$ on the compact set $[0,T]\times \cball{\setR^n}{M_1}\times [0,M_1]$ yields
that there exists $\delta_0>0$ such that
\begin{equation}\label{e236}
g(t,u,w)\geq A_0/2>0,\qquad t\in[0,T],\ u\in \cball{\setR^n}{M_1},\ 0\leq w <\delta_0.
\end{equation}
Let $j_0h$ be the last mesh point less than $\alpha^*_h$. Then $\alpha^*_h-j_0h\leq h$, and
\eref{e232} yields $\tau_h(j_0h)>0$.
Define $h^*=\min\{\delta_0/M_3,h_0\}$, where $h_0$ is defined by \eref{e261}.
Using relation \eref{e235} and the definition of $\alpha^*_h$ we get
$$
0<\tau_h(j_0h)=\tau_h(j_0h)-\tau_h(\alpha^*_h)\leq M_3(\alpha^*_h-j_0h)\leq M_3h<\delta_0,\qquad 0<h<h^*.
$$
But then \eref{e236} implies
$$
\dot \tau_h(j_0h)=g(j_0h,x_h(j_0h),\tau_h(j_0h))>0,
$$
which contradicts to the selection of $j_0h$. This contradiction means that $\tau_h(\alpha^*_h)=0$ cannot
happen. Therefore $\alpha^*_h=\alpha$, and  $(x_h,\tau_h)$ exists on
$[0,\alpha]$, moreover \eref{e414a}, \eref{e414}, \eref{e415} and \eref{e415b} hold.
\end{proof}
\bigskip

For the rest of this section we use the following notation. Let $\alpha\not\in\calT$ be defined by Lemma~\ref{l9}, and let
$k_0\in\{0,1,\ldots,K\}$ be the largest index such that $t_{k_0}<\alpha$, $\ell_0\in\setN$ be the  index
such that $-\lambda\in[t_{-\ell_0},t_{-\ell_0+1})$, and redefine $t_{-\ell_0}$, $\calI_{-\ell_0}$,
$t_{k_0+1}$ and  $\calI_{k_0}$ in the following way:
\begin{equation}\label{e313}
t_{-\ell_0}=-\lambda,\quad  \calI_k=[t_k,t_{k+1}),\ \, k=-\ell_0,-\ell_0+1,\ldots,{k_0}-1,
\quad \mbox{and}\quad t_{{k_0}+1}=\alpha,\quad \calI_{k_0}=[t_{k_0},\alpha].
\end{equation}

Next we formulate our existence and uniqueness result.
First we prove the result under a condition which implies that the delayed time function $t-\tau(t)$ is strictly monotone increasing.
(See also Theorem 3.3 in \cite{Ha2024} for a related statement.)
For the proof of the uniqueness we apply the technique used in \cite{C2022}.

\begin{theorem}\label{t1}
Assume (A1)--(A4),  let $\lambda>0$,
and let  $\alpha>0$, $h^*>0$, $M_1,M_2$ and $M_3$ be defined by Lemma~\ref{l9}, moreover
\begin{equation}\label{e851}
 g(t,u,w)<1,\qquad t\in[0,\alpha], \quad u\in\cball{\setR^n}{M_1},\quad w \in[0,M_1].
\end{equation}
 Then the \IIVP{e1}{e2} has a unique solution $(x,\tau)$ on $(-\infty,\alpha]$.
Moreover,
\begin{align}
\tau(t) &> 0,\qquad t\in[0,\alpha],\label{e314a}\\
|x(t)|+\tau(t)&\leq M_1,\qquad t\in(-\infty,\alpha],\label{e314}\\
|x(t)-x(\bar t)|&\leq M_2|t-\bar t|,\qquad  t,\bar t\in\calI_k,\quad k=-\ell_0,\ldots,k_0,\quad
\label{e315}\\
|\tau(t)-\tau(\bar t)|&\leq M_3|t-\bar t|,\qquad  t,\bar t\in[0,\alpha].\label{e315b}
\end{align}
\end{theorem}
\begin{proof}
 Let  $(x_h,\tau_h)$ be defined
 by the \IIVP{ae1}{ae2} for $0<h<h^*$. Note that $(x_h(t),\tau_h(t))=(\phi(t),\lambda)$ is independent of $h$ on $(-\infty,0]$.
 It follows from Lemmas~\ref{l10}, \ref{l9} and Arzel\`a--Ascoli Theorem
 that there exists a sequence $h_i$ and functions $x\in\PC$ and $\tau\in C([0,\alpha],\setR)$
 such that $h_i\to0$ as $i\to\infty$ and
 \begin{equation}\label{e854}
 \lim_{i\to\infty}|x_{h_i}-x|_\PC=0\qquad\mbox{and}\qquad  \lim_{i\to\infty}|\tau_{h_i}-\tau|_{C([0,\alpha],\setR)}=0.
 \end{equation}
 Note that in the above relations the restrictions of $x_{h_i}$ and $\tau_{h_i}$ to $[-M_1,\alpha]$ and  $[0,\alpha]$
 are denoted simply by $x_{h_i}$ and $\tau_{h_i}$, respectively.
 Consider equation  \eref{e282} for $h=h_i$, and taking the limit $i\to\infty$, it is easy to check that
 $x$ and $\tau$ satisfy
$$
 \tau(t) = \tau(t_k)+\int_{t_k}^t g(s,x(s),\tau(s))\,ds,\qquad t\in\calI_k,\quad k=0,1,\ldots,k_0.
$$
Then $x$ and $\tau$ fulfill \eref{e1b} on $[0,\alpha]$. Relations \eref{e414} and \eref{e854} entail that
$|x(t)|+|\tau(t)|\leq M_1$ for $t\in[0,\alpha]$, hence
 $x(t)\in\cball{\setR^n}{M_1}$ and $\tau(t)\in[0,M_1]$ for $t\in[0,\alpha]$.
Since $g$ is continuous on the compact set
$[0,\alpha]\times\cball{\setR^n}{M_1}\times[0,M_1]$, relation \eref{e851} yields
$$
\essinf_{t\in[0,\alpha]} \frac{d}{dt} (t-\tau(t))=1-\esssup_{t\in[0,\alpha]}g(t,x(t),\tau(t))>0,
$$
i.e., the delayed time function is strictly monotone increasing on $[0,\alpha]$.
Therefore the set
\begin{equation}\label{e855}
\calU:=\bigcup_{k=-\ell_0}^{k_0}\{s\in[0,\alpha]\sth s-\tau(s)=t_k\}
\end{equation}
has Lebesgue measure 0. From the continuity of $\tau$, \eref{e216} and \eref{e854} we deduce
$$
[s]_{h_i}-[\tau_{h_i}([s]_{h_i})]_{h_i}\to s-\tau(s),\qquad \mbox{for}\ s\in[0,\alpha],
$$
and since $x$ has jump discontinuity only at $t\in\calT^*$, we get
\begin{equation}\label{e853}
 x_{h_i}([s]_{h_i})\to x(s)\quad\mbox{and}\quad
 x_{h_i}([s]_{h_i}-[\tau_{h_i}([s]_{h_i})]_{h_i})\to x(s-\tau(s))\qquad \mbox{for a.e.\ } s\in[0,\alpha]
\end{equation}
as $i\to\infty$.
From \eref{e281}, using \eref{e216}, \eref{e854}, \eref{e853} and  Lebesgue's dominated convergence theorem, we obtain
$$
x(t) = x(t_k)+\int_{t_k}^t f(s,x(s),x(s-\tau(s)))\,ds,\qquad t\in\calI_k,\quad k=0,1,\ldots,k_0.
$$

Let $k\in\{1,\ldots,k_0\}$.
Using the uniform convergence of $x_{h_i}$ to $x$ on $\calI_{k-1}$, we have (see, e.g., Theorem 7.11 in \cite{Ru1976})
$$
\lim_{i\to\infty} x_{h_i}(t_k-)=\lim_{i\to\infty}\lim_{t\to t_k-} x_{h_i}(t)=\lim_{t\to t_k-}\lim_{i\to\infty} x_{h_i}(t)=x(t_k-).
$$
Therefore the jump condition \eref{ae1c} and the continuity of $I_k$ imply
\begin{align*}
 \Delta x(t_k)
 &= x(t_k)-x(t_k-)\\
 &= \lim_{i\to\infty}\Bigl(x_{h_i}(t_k)-x_{h_i}(t_k-)\Bigr)\\
 &= \lim_{i\to\infty}I_k(x_{h_i}(t_k-))\\
 &= I_k(x(t_k-)),
\end{align*}
hence $x$ satisfies \eref{e1c}.

Clearly, $x$ satisfies \eref{e2a} on $[-M_1,0]$, and $\tau(0)=\lambda$.
Using \eref{e2a} and \eref{e2} we can extend the definition of $x$ and $\tau$ to
the infinite interval $(-\infty,0]$.
Then $(x,\tau)$ is a solution of the \IIVP{e1}{e2} on $(-\infty,\alpha]$.
To show \eref{e314a} suppose there exists  $t^*\in(0,\alpha]$ such that
$$
\tau(t)>0,\qquad t\in[0,t^*)\qquad\mbox{and}\qquad \tau(t^*)=0.
$$
Then $\dot\tau(t^*-)\leq 0$, but it contradicts to (A2) (ii). Therefore \eref{e314a} holds.
Since $x_{h_i}$ and $\tau_{h_i}$ satisfy \eref{e414}--\eref{e415b}, taking the limit $i\to\infty$ yields
 \eref{e314}--\eref{e315b}.

To prove the uniqueness, let $(x,\tau)$ be the solution obtained by the previous argument, and
assume that $(\bar x,\bar \tau)$ is an other solution of the \IIVP{e1}{e2} on $(-\infty,\alpha]$
(which also satisfies \eref{e314a} because of (A2) (ii), but does not necessarily satisfy \eref{e314}--\eref{e315b}).
Let $M_1^*\geq M_1$ be such that
\begin{equation}\label{e260}
|\bar x(t)|+\bar\tau(t)\leq M_1^*,\quad t\in[0,\alpha].
\end{equation}
It follows from the monotonicity of the delayed time function that $s-\bar\tau(s)\geq-\lambda$ for $s\in[0,\alpha]$.
Let $L_1=L_1(M_1^*)$, $L_2=L_2(M_1^*)$  and $L_4=L_4(\lambda)$
be the Lipschitz constants from (A1) (ii), (A2) (iii)  and (A4) (i).
Let us introduce
$$
\alpha_1:=\inf\{t\in[0,\alpha]\sth |x(t)-\bar x(t)|+|\tau(t)-\bar\tau(t)|\neq0\}.
$$
Suppose $\alpha_1<\alpha$.
Let $k\in\{0,\ldots,k_0\}$ be such that $\alpha_1\in\calI_k$, and
let $\delta_0>0$ be such that $[\alpha_1,\alpha_1+\delta_0]\subset\calI_k$.
Then it is easy to obtain for $t\in[\alpha_1,\alpha_1+\delta_0]$
\begin{align*}
 |x(t)-\bar x(t)|
 &\leq |x(t_k)-\bar x(t_k)|
 +\int_{t_k}^t \Bigl|f(s,x(s),x(s-\tau(s)))-f(s,\bar x(s),\bar x(s-\bar\tau(s))) \Bigr|ds\\
 &=\int_{\alpha_1}^t \Bigl|f(s,x(s),x(s-\tau(s)))-f(s,\bar x(s),\bar x(s-\bar\tau(s))) \Bigr|ds\\
 &\leq  \int_{\alpha_1}^t L_1\Bigl(|x(s)-\bar x(s)|+|x(s-\tau(s))-x(s-\bar\tau(s))|
   +|x(s-\bar \tau(s))-\bar x(s-\bar\tau(s))| \Bigr)ds.\\
\end{align*}
The positivity of $\tau$ and $\bar\tau$ on $[0,\alpha]$ guarantees that there exists $\delta_1>0$ such that
$\tau(t)\geq\delta_1$ and $\bar\tau(t)\geq\delta_1$ for $t\in[0,\alpha]$.
Suppose $\alpha_1-\tau(\alpha_1)=\alpha_1-\bar\tau(\alpha_1)\in\calI_j$ for some $j\in\{-\ell_0,\ldots,k\}$. Then there exists $\epsilon>0$ such that
$$
J:=[\alpha_1-\tau(\alpha_1),\alpha_1-\tau(\alpha_1)+\epsilon]\subset\calI_j.
$$
We have that both delayed time functions $t-\tau(t)$ and $t-\bar\tau(t)$ are increasing  and continuous, hence
there exist $0<\delta<\min\{\delta_0,\delta_1\}$ such that
$$
s-\tau(s)\in J\qquad\mbox{and}\qquad s-\bar\tau(s)\in J\qquad\mbox{for}\ s\in[\alpha_1,\alpha_1+\delta].
$$
Therefore,
 for $t\in[\alpha_1,\alpha_1+\delta]$ we obtain $|x(s-\bar \tau(s))-\bar x(s-\bar\tau(s))|=0$, and
\begin{align*}
 |x(t)-\bar x(t)|
 &\leq  \int_{\alpha_1}^t L_1\Bigl(|x(s)-\bar x(s)|+M_2|\tau(s)-\bar\tau(s)|\Bigr)ds.
\end{align*}
Similarly, for $t\in[\alpha_1,\alpha_1+\delta]$
\begin{align*}
 |\tau(t)-\bar \tau(t)|
 &\leq |\tau(t_k)-\bar \tau(t_k)|
 +\int_{t_k}^t \Bigl|g(s,x(s),\tau(s))-g(s,\bar x(s),\bar\tau(s)) \Bigr|ds\\
 &=\int_{\alpha_1}^t \Bigl|g(s,x(s),\tau(s))-g(s,\bar x(s),\bar\tau(s)) \Bigr|ds\\
 &\leq \int_{\alpha_1}^t L_2\Bigl(|x(s)-\bar x(s)|+|\tau(s)-\bar\tau(s)|  \Bigr)ds.
\end{align*}
We define
\begin{align*}
 L_0&=L_1\max\{1,M_2\}+L_2,\\
 \omega(t) &=  |x(t)-\bar x(t)|+ |\tau(t)-\bar \tau(t)|,\qquad t\in(-\infty,\alpha].
\end{align*}
Then adding the previous estimates we get for $t\in[\alpha_1,\alpha_1+\delta]$
$$
 \omega(t)
 \leq \int_{\alpha_1}^t L_0\sup_{-M_1\leq\zeta\leq s}\omega(\zeta)\,ds.
$$
Using $\omega(\alpha_1)=0$, Gronwall's lemma implies that $\omega(t)=0$ for $t\in[\alpha_1,\alpha_1+\delta]$,
 which contradicts to the definition of $\alpha_1$. So we obtain
 $\omega(t)=0$ for $t\in[0,\alpha]$, and the uniqueness of
 the solution follows.
\end{proof}
\bigskip

The proofs of Lemma~\ref{l9} and Theorem~\ref{t1} imply immediately the following result.

\begin{corollary}\label{cor2}
Let $\lambda>0$, and assume (A1)--(A4) where $f$ and $g$ are globally Lipschitz
continuous on their domain, i.e., $L_1$ and $L_2$ do not depend on $M$ in (A1) (ii) and (A2) (iii),
respectively, moreover
$$
g(t,u,w)<1,\qquad t\in[0,T],\quad u\in\setR^n,\quad w\in[0,\infty).
$$
Then the \IIVP{e1}{e2} has a unique solution $(x,\tau)$ on $(-\infty,T]$, and relations \eref{e314a}--\eref{e315b} hold
with $\alpha=T$.
\end{corollary}
\smallskip

\section{Numerical approximation}\label{s4}
\setcounter{equation}0
\setcounter{theorem}0

Now we  study numerical approximation of solutions of  the \IIVP{e1}{e2}
with the help of EPCAs. The next theorem shows that the solutions of
\eref{ae1}--\eref{ae2} approximate that of \eref{e1}--\eref{e2} uniformly
on the compact interval $[0,\alpha]$. The key assumption for this result is the
piecewise strict monotonicity of the delayed time function $t-\tau(t)$.
\medskip

\begin{theorem}\label{t2}
Suppose (A1)-(A4) hold,  let $\lambda>0$,  and let $0<\alpha\leq T$ and $h^*>0$ be defined by Lemma~\ref{l9}.
Let $(x,\tau)$ be any solution of the \IIVP{e1}{e2} on $(-\infty,\alpha]$,
where the corresponding delayed time function
$t-\tau(t)$ is piecewise strictly monotone  on $[0,\alpha]$, and suppose the impulsive time moments $t_k$ ($k=-\ell_0,\ldots,k_0$) are not local extreme values or extreme points of the delayed time function, and $\dot\tau(0)\neq1$ and $\dot\tau(\alpha-)\neq1$. Let $(x_h,\tau_h)$ be the solution of the \IIVP{ae1}{ae2} on $(-\infty,\alpha]$ for $0<h<h^*$.
Then
 \begin{equation}\label{e221}
  \lim_{h\to0+}\max_{t\in[0,\alpha]}|x_h(t)-x(t)|=0\qquad\mbox{and}\qquad \lim_{h\to0+}\max_{t\in[0,\alpha]}|\tau_h(t)-\tau(t)|=0.
 \end{equation}
\end{theorem}
\begin{proof}
Let $M_1, M_2$ and $M_3$ be defined by Lemma~\ref{l9}.
We define
$$
M_1^*:=\max\Bigl\{\sup\{|x(t)|+|\tau(t)|\sth t\in[0,\alpha]\},M_1\Bigr\}.
$$
Then $t-\tau(t)\geq -M_1^*$ for $t\in[0,\alpha]$.
Let $M_2^*$ and $M_3^*$ be defined by \eref{e856} and \eref{e857}, respectively, where $M_1$ is replaced with $M_1^*$,
and let $L_1=L_1(M_1^*)$, $L_2=L_2(M_1^*)$,  $L_3$
and $L_4=L_4(M_1^*)$ be the Lipschitz constants from (A1)--(A4), respectively.
Note that $M_2^*\geq M_2$ and $M_3^*\geq M_3$.
We use the notations defined by \eref{e313} with the modification that let $\ell_0$ be the index for which
$-M_1^*\in[t_{-\ell_0},t_{-\ell_0+1})$, and redefine $t_{-\ell_0}=-M_1^*$.
For $h\in(0,h^*)$ we define the constants
\begin{align*}
 \mu_{1,h} &= \max_{k=0,\ldots,k_0}\int_{t_k}^{t_{k+1}} \Bigl|f([s]_h,x(s),x(s-\tau(s)))-f(s,x(s),x(s-\tau(s)))\Bigr|\,ds,\\
 \mu_{2,h} &= \max_{k=0,\ldots,k_0}\int_{t_k}^{t_{k+1}} \Bigl|g([s]_h,x(s),\tau(s))-g(s,x(s),\tau(s))\Bigr|\,ds.
\end{align*}
Relation \eref{e216}, the continuity of $f$, $g$, $x$ and $\tau$ on the intervals $\calI_k$ and  Lebesgue's dominated convergence theorem imply that
\begin{equation}\label{e212}
\lim_{h\to0+}(\mu_{1,h}+\mu_{2,h})=0.
\end{equation}
Since $x_h$ and $\tau_h$ are continuous on $\calI_k$, \eref{ae1} and \eref{ae1b} yield for $t\in\calI_k$
and $k=0,1,\ldots,k_0$
\begin{align}
 x_h(t) &= x_h(t_k)+\int_{t_k}^t f([s]_h,x_h([s]_h),x_h([s]_h-[\tau_h([s]_h)]_h))\,ds,\label{e204}\\
 \tau_h(t) &= \tau_h(t_k)+\int_{t_k}^t g([s]_h,x_h([s]_h),\tau_h([s]_h))\,ds.\label{e205}
\end{align}
Integrating \eref{e1} and \eref{e1b} from $t_k$ to $t\in\calI_k$ we get
\begin{align}
 x(t) &= x(t_k)+\int_{t_k}^t f(s,x(s),x(s-\tau(s)))\,ds,\label{e204b}\\
 \tau(t) &= \tau(t_k)+\int_{t_k}^t g(s,x(s),\tau(s))\,ds.\label{e205b}
\end{align}
Taking the difference of \eref{e204} and \eref{e204b}, and using assumption (A1)(ii) and the definition of $\mu_{1,h}$ we get
\begin{align}
  |x_h(t)-x(t)| &\leq |x_h(t_k)-x(t_k)|
  +\int_{t_k}^t \Bigl|f([s]_h,x_h([s]_h),x_h([s]_h-[\tau_h([s]_h)]_h))-f(s,x(s),x(s-\tau(s)))\Bigr|\,ds\nl
  &\leq |x_h(t_k)-x(t_k)|+\int_{t_k}^{t} \Bigl|f([s]_h,x(s),x(s-\tau(s)))-f(s,x(s),x(s-\tau(s)))\Bigr|\,ds\nl
  &\quad +\int_{t_k}^{t} \Bigl|f([s]_h,x_h([s]_h),x_h([s]_h-[\tau_h([s]_h)]_h))-f([s]_h,x(s),x(s-\tau(s)))\Bigr|\,ds\nl
  &\leq |x_h(t_k)-x(t_k)|+\mu_{1,h}\nl
  &\quad +\int_{t_k}^{t} L_1\Bigl(|x_h([s]_h)-x(s)|+|x_h([s]_h-[\tau_h([s]_h)]_h))-x(s-\tau(s))|\Bigr)\,ds\label{e209}
\end{align}
for $t\in\calI_k$, $k=0,\ldots,k_0$.
We introduce further notations
\begin{align*}
 z_h(t) &=|x_h(t)-x(t)|,\qquad t\in(-\infty,\alpha],\\
 \eta_h(t) &= |\tau_h(t)-\tau(t)|,\qquad t\in(-\infty,\alpha],\\
 \omega_h(t) &= z_h(t)+\eta_h(t),\qquad t\in(-\infty,\alpha].
\end{align*}
Then \eref{e209} combined with relations \eref{e216}, \eref{e415}  and \eref{e415b} implies
\begin{align}
  z_h(t) &\leq z_h(t_k)+\mu_{1,h}
   +\int_{t_k}^{t} L_1\Bigl(|x([s]_h)-x(s)|+|x([s]_h-[\tau_h([s]_h)]_h)-x(s-\tau(s))|\Bigr)\,ds\nl
  &\quad +\int_{t_k}^{t} L_1\Bigl(z_h([s]_h)+z_h([s]_h-[\tau_h([s]_h)]_h)\Bigr)\,ds\label{e210a}
\end{align}
for $t\in\calI_k$, $k=0,\ldots,k_0$. Suppose $t<t_k+h$. Then
$$
\int_{t_k}^{t} \Bigl(|x([s]_h)-x(s)|+|x([s]_h-[\tau_h([s]_h)]_h)-x(s-\tau(s))|\Bigr)\,ds\leq 4M_1^*h.
$$
For $t\in(t_k+h,t_{k+1})$ we have $[t]_h\geq t_k$. Hence, applying the definition of $M_1^*$, $M_2^*$
and \eref{e216}, we get
\begin{equation}\label{e860}
 \int_{t_k}^{t} |x([s]_h)-x(s)|\,ds
 =\int_{t_k}^{t_k+h} |x([s]_h)-x(s)|\,ds+\int_{t_k+h}^{t} |x([s]_h)-x(s)|\,ds
 \leq 2M_1^*h+M_2^*h\alpha.
\end{equation}
Note that \eref{e860} holds for all $t\in\calI_k$.

Fix $t\in(t_k+h,t_{k+1})$. Next we estimate the integral with the delayed terms
$$
\int_{t_k}^{t} |x([s]_h-[\tau_h([s]_h)]_h)-x(s-\tau(s))|\,ds.
$$
Let $\calU$ be defined by \eref{e855}, and let $s_0=0<s_1<\cdots<s_m=\alpha$
be the points of local extrema of the delayed time function $t-\tau(t)$, and define $\calM=\{s_1,\ldots,s_{m-1}\}$.
Because of the assumption of the theorem,
$\calU$ and $\calM$ are disjoint sets. Let $d_0=\min\{s_{j+1}-s_j\sth j=0,\ldots,m-1\}$,
and $0<d_1<d_0$  be the smallest distance between consecutive points of the set
$\calU\cup\calM\cup\{0,\alpha\}$.
Let us define $\calM^+$ and $\calM^-$ as the set of points $p_j$ of $\calM\cup\{0\}$ for which
the delayed time function is increasing and decreasing, respectively, on the intervals $(p_j,p_{j+1})$.

Let $0<\delta^*<d_1/4$. Then the intervals $(s_j-\delta^*,s_j+\delta^*)$
do not contain any point of $\calU\cap(0,\alpha)$
for all $s_j\in\calM\cup\{0,\alpha\}$. Define the constants
\begin{align*}
\epsilon_0&=\essinf\biggl\{|1-\dot\tau(s)|\sth s\in(s_0,s_1-\delta^*)\cup
\bigcup_{j=1}^{m-2}(s_j+\delta^*,s_{j+1}-\delta^*) \cup(s_{m-1}+\delta^*,s_m)\biggr\},\\
N_0&=\esssup\Bigl\{|1-\dot\tau(s)|\sth s\in(0,\alpha)\Bigr\}.
\end{align*}
The definition of piecewise strict monotonicity,  the assumption $\dot\tau(0)\neq1$ and $\dot\tau(\alpha-)\neq1$,
and \eref{e1b} imply that $\epsilon_0$ and $N_0$ are finite positive numbers.
The assumed piecewise monotonicity of
the delayed time function yields that  $\calU$ has finitely many elements. Let
$u_{k,1}<u_{k,2}<\cdots<u_{k,i_k}$ be the elements of
$\calU\cap(t_k,t_{k+1})$, moreover define $u_{k,0}=t_k$ and $u_{k,i_k+1}=t_{k+1}$. Then for every $k\in\{0,\ldots,k_0\}$ and $i\in\{0,\ldots,i_k\}$ there exists
$j_{k,i}\in\{-\ell_0,\ldots,k_0\}$ such that $s-\tau(s)\in\calI_{j_{k,i}}$ for $s\in(u_{k,i},u_{k,i+1})$.

It follows from \eref{e216} and the definition of $M_3^*$ that
\begin{align}
|([s]_h-[\tau_h([s]_h)]_h)-(s-\tau(s))|
&\leq |[s]_h-s|+|[\tau_h([s]_h)]_h-\tau_h([s]_h)|+|\tau_h([s]_h)-\tau_h(s)|+|\tau_h(s)-\tau(s)|\nl
&\leq Mh+\eta_h(s),\qquad s\in[0,\alpha],\label{e863}
\end{align}
where  $M=2+M_3^*$.

%The assumption of the theorem yields that the sets $\calT^*$ and $\{t_k-\tau(t_k)\sth k=0,\ldots,k_0\}$ are disjoint.
Let $e_0^*$ be the smallest distance between consecutive elements of the set
$\calT^*\cup\{t_k-\tau(t_k)\sth k=0,\ldots,k_0\}$,
$e_1^*$ be the smallest distance between consecutive elements of the set
$\calT^*\cup\{s_j-\tau(s_j)\sth j=1,\ldots,m-1\}$,
and
we define $\delta$ so that $0<\delta<\min\{\delta^*,e_0^*/(4\epsilon_0),e_0^*/(N_0+\epsilon_0),e_1^*/\epsilon_0\}$.
Then the intervals $(u_{k,i},u_{k,i}+\delta)$ and $(u_{k,i+1}-\delta,u_{k,i+1})$ do not contain
any element of $\calU\cup\calM$,
and the delayed time function is either monotone increasing or decreasing on these intervals.

We define the constants
\begin{align*}
n_0&=\max\{ i_k\sth k=0,\ldots,k_0\}+1,\\
K_0&=\frac{1+L_2(1+M_2^*+M_3^*)\delta}{\epsilon_0},\\
K_1&=2M_1^*n_0(2K_0+1)+\alpha M_2^*M,\\
K_2&=L_1(2M_1^*+M_2^*\alpha+K_1),\\
K_3&=L_2(2M_1^*+M_2^*\alpha+M_3^*\alpha),\\
K_4&=\max\Bigl\{\frac{2M_1^*}{\epsilon_0},1\Bigr\},\\
D_h&=\mu_{1,h}+\mu_{2,h}+(K_2+K_3)h,\\
L_0&= \max\Bigl\{L_1\max\Bigl\{1,\frac{L_2}{\epsilon_0}\Bigr\}+L_2,L_1M_2^*,L_1\Bigr\},\\
m^*&= 2(i_0+\cdots+i_{k_0}),\\
E^*&=2^{m^*+k_0}K_4^{m^*+k_0}\Bigl(\sum_{i=0}^{k_0}(1+L_3)^i\Bigr)^{m^*+2k_0}(e^{3L_0 \alpha})^{m^*+2k_0}.
\end{align*}
It follows from relation \eref{e212} that there exists $h_1\in(0,h^*)$ such that
\begin{equation}\label{e873}
 D_hE^* <\frac{\epsilon_0\delta}{2},\qquad 0<h<h_1.
\end{equation}

Let $h_2=\min\{h_1,\frac{\epsilon_0\delta}{2M},\delta\}$ and $h\in(0,h_2)$. Since $\eta_h(0)=0$, for small
$t$ we have $\eta_h(t)<\frac{\epsilon_0 \delta}{2}$.
Suppose  $0<\beta \leq \alpha$ is such that
\begin{equation}\label{e865}
\eta_h(s)<\frac{\epsilon_0 \delta}{2},\qquad s\in[0,\beta),\qquad \eta_h(\beta)=\frac{\epsilon_0 \delta}{2}.
\end{equation}
Now we have the following observations. If the delayed time function is increasing on $(u_{k,i},u_{k,i}+\delta)$, then
there is no $s\in(u_{k,i},u_{k,i}+\delta)\cap[0,\beta]$  such that
$[s]_h-[\tau_h([s]_h)]_h\geq t_{j_{k,i}+1}$. If $i>0$, then $t_{j_{k,i}}=u_{k,i}-\tau(u_{k,i})$.
In fact, if the inequality were true, then the estimates
$$
e_0^*\leq t_{j_{k,i}+1}-t_{j_{k,i}}\leq [s]_h-[\tau_h([s]_h)]_h-(u_{k,i}-\tau(u_{k,i}))\leq
s-\tau(s)+\epsilon_0\delta-(u_{k,i}-\tau(u_{k,i}))\leq (N_0+\epsilon_0)\delta
$$
would contradict  the definitions of $s$, $\delta$ and $h$, \eref{e863} and \eref{e865}.
For $i=0$ the definition of $e_0^*$ yields the contradiction
$$
e_0^*\leq t_{j_{k,i}+1}-(t_{k}-\tau(t_{k}))\leq [s]_h-[\tau_h([s]_h)]_h-(t_{k}-\tau(t_{k}))\leq
s-\tau(s)+\epsilon_0\delta-(t_{k}-\tau(t_{k}))\leq (N_0+\epsilon_0)\delta.
$$
In a similar manner we can check that
if the delayed time function is decreasing on $(u_{k,i},u_{k,i}+\delta)$, then
there is no $s\in(u_{k,i},u_{k,i}+\delta)\cap[0,\beta]$  such that
$[s]_h-[\tau_h([s]_h)]_h< t_{j_{k,i}}$.
If the delayed time function is increasing on $(u_{k,i+1}-\delta,u_{k,i+1})$, then
there is no $s\in(u_{k,i+1}-\delta,u_{k,i+1})\cap[0,\beta]$  such that
$[s]_h-[\tau_h([s]_h)]_h< t_{j_{k,i}}$. Finally, if the delayed time function is decreasing on $(u_{k,i+1}-\delta,u_{k,i+1})$, then  there is no $s\in(u_{k,i+1}-\delta,u_{k,i+1})\cap[0,\beta]$  such that
$[s]_h-[\tau_h([s]_h)]_h\geq t_{j_{k,i}+1}$.

We recall the notation that $s-\tau(s)\in \calI_{j_{k,i}}=[t_{j_{k,i}},t_{j_{k,i}+1})$ for $s\in(u_{k,i},u_{k,i+1})$.
We define the disjoint sets
\begin{align*}
A_{k,i}^h&=\{s\in(u_{k,i},u_{k,i}+\delta)\sth [s]_h-[\tau_h([s]_h)]_h<t_{j_{k,i}}\quad\mbox{or}\quad
[s]_h-[\tau_h([s]_h)]_h\geq t_{j_{k,i}+1}\},\\
B_{k,i}^h&=\{s\in(u_{k,i},u_{k,i+1})\sth [s]_h-[\tau_h([s]_h)]_h\in\calI_{j_{k,i}}\},\\
C_{k,i}^h&=\{s\in(u_{k,i+1}-\delta,u_{k,i+1})\sth [s]_h-[\tau_h([s]_h)]_h<t_{j_{k,i}}\quad\mbox{or}\quad
[s]_h-[\tau_h([s]_h)]_h\geq t_{j_{k,i}+1}\}.
\end{align*}

Next we show that
\begin{equation}\label{e875}
(u_{k,i}+\delta,u_{k,i+1}-\delta)\subset B_{k,i}^h.
\end{equation}
We have four cases: (i)  $t_{j_{k,i}}=u_{k,i}-\tau(u_{k,i})$ and $t_{j_{k,i+1}}=u_{k,i+1}-\tau(u_{k,i+1})$; (ii)   $t_{j_{k,i+1}}=u_{k,i}-\tau(u_{k,i})$ and $t_{j_{k,i}}=u_{k,i+1}-\tau(u_{k,i+1})$; (iii) $t_{j_{k,i}}=u_{k,i}-\tau(u_{k,i})=
u_{k,i+1}-\tau(u_{k,i+1})$, and $s-\tau(s)>t_{j_{k,i}}$ for $s\in(u_{k,i},u_{k,i+1})$; and (iv)
$t_{j_{k,i+1}}=u_{k,i}-\tau(u_{k,i})=u_{k,i+1}-\tau(u_{k,i+1})$, and $s-\tau(s)<t_{j_{k,i+1}}$ for $s\in(u_{k,i},u_{k,i+1})$.
The definition of $e_0^*$ yields that $t_{j_{k,i+1}}-t_{j_{k,i}}\geq e_0^*>
4\epsilon_0\delta$.

In case (i) first suppose that the delayed time function is monotone increasing on $(u_{k,i},u_{k,i+1})$.
The definition of $\epsilon_0$ yields that
$$
(u_{k,i}+\delta)-\tau(u_{k,i}+\delta)-(u_{k,i}-\tau(u_{k,i}))\geq \epsilon_0\delta,
$$
hence $s-\tau(s)>t_{j_{k,i}}+\epsilon_0\delta$ for $s\in(u_{k,i}+\delta,u_{k,i+1})$.
Similarly, we get that $s-\tau(s)<t_{j_{k,i+1}}-\epsilon_0\delta$ for $s\in(u_{k,i},u_{k,i+1}-\delta)$.
If the delayed time function is not monotone increasing, then it has local extrema between $u_{k,i}$ and $u_{k,i+1}$.
The definitions of $e_1^*$ and $\delta$ imply for any local maximum point $s_1$ that
$s_1-\tau(s_1)\leq t_{j_{k,i+1}}-e_1^*<t_{j_{k,i+1}}-\epsilon_0\delta$, and for any local minimum point $s_2$,
it follows $s_2-\tau(s_2)> t_{j_{k,i}}+\epsilon_0\delta$. Therefore, in both cases, we have
$t_{j_{k,i}}+\epsilon_0\delta<s-\tau(s)<t_{j_{k,i+1}}-\epsilon_0\delta$ for $s\in(u_{k,i}+\delta,u_{k,i+1}-\delta)$.
Then for $s\in(u_{k,i}+\delta,u_{k,i+1}-\delta)\cap[0,\beta]$ it follows from \eref{e863},  \eref{e865}
and the above estimates
that
$$
t_{j_{k,i}}< s-\tau(s)-\epsilon_0\delta\leq [s]_h-[\tau_h([s]_h)]_h\leq s-\tau(s)+\epsilon_0\delta<  t_{j_{k,i}+1},
$$
which proves \eref{e875} in case (i). In case (ii) the proof of relation \eref{e875} is similar.

In case (iii) the delayed time function has a  maximal value at a point $s^*\in(u_{k,i},u_{k,i+1})$.
The definitions of $e_1^*$ and $\delta$ imply that
$t_{j_{k,i+1}}-\epsilon_0\delta>t_{j_{k,i+1}}-e_1^*\geq s^*-\tau(s^*)$.
Then, as in case (i), it is easy to obtain that
$s-\tau(s)>t_{j_{k,i}}+\epsilon_0\delta$ for $s\in(u_{k,i}+\delta,u_{k,i+1}-\delta)$, and hence for such $s$ we get
$$
t_{j_{k,i}}< s-\tau(s)-\epsilon_0\delta\leq [s]_h-[\tau_h([s]_h)]_h\leq s-\tau(s)+\epsilon_0\delta
\leq s^*-\tau(s^*)+\epsilon_0\delta  <t_{j_{k,i}+1}.
$$
This proves \eref{e875} in case (iii). The proof of \eref{e875} for case (iv) is similar.

Next we estimate the Lebesgue measure of $A_{k,i}^h$.
Let $v_{k,i}^h$ be the largest element of $A_{k,i}^h\cap [t_k,t]\cap\setN_0h$.
Then $A_{k,i}^h\cap [t_k,t]\subset (u_{k,i},v_{k,i}^h+h)$.
We have two cases: either (i) $u_{k,i}\in(s_j+\delta,s_{j+1}-\delta)$ for some $s_j\in\calM^+$, or (ii)
$u_{k,i}\in(s_j+\delta,s_{j+1}-\delta)$ for some $s_j\in\calM^-$. In case (i)
the definition of $A_{k,i}^h$ implies
$$
v_{k,i}^h-\tau_h(v_{k,i}^h)\leq v_{k,i}^h-[\tau_h(v_{k,i}^h)]_h<t_{j_{k,i}}\leq u_{k,i}-\tau(u_{k,i}),
$$
therefore we obtain from the definition of $\epsilon_0$ and the monotonicity of the delayed time function that
\begin{align*}
\epsilon_0(v_{k,i}^h-u_{k,i})\leq v_{k,i}^h-\tau(v_{k,i}^h)-(u_{k,i}-\tau(u_{k,i}))<\tau_h(v_{k,i}^h)-\tau(v_{k,i}^h).
\end{align*}
In case (ii)  the delayed time function is decreasing on $(u_{k,i},u_{k,i}+\delta)$, and
$$
v_{k,i}^h-(\tau_h(v_{k,i}^h)-h)> v_{k,i}^h-[\tau_h(v_{k,i}^h)]_h\geq t_{j_{i+1}}\geq u_{k,i}-\tau(u_{k,i}),
$$
we get
\begin{align*}
-\epsilon_0(v_{k,i}^h-u_{k,i})\geq v_{k,i}^h-\tau(v_{k,i}^h)-(u_{k,i}-\tau(u_{k,i}))
>\tau_h(v_{k,i}^h)-\tau(v_{k,i}^h)-h.
\end{align*}
Therefore in both cases the Lebesgue measure of $A_{k,i}^h$ can be estimated as
\begin{equation}\label{e867a}
m\Bigl(A_{k,i}^h\cap [t_k,t]\Bigr)\leq \frac{1}{\epsilon_0} |\tau_h(v_{k,i}^h)-\tau(v_{k,i}^h)|+\left(\frac{1}{\epsilon_0}+1\right)h,\qquad i=0,\ldots,i_{k,t}.
\end{equation}
Using that $\tau_h$ and $\tau$ are absolutely continuous, we get from \eref{e1b} and \eref{ae1b} that
\begin{align*}
 |\tau_h(v_{k,i}^h)-\tau(v_{k,i}^h)|
 &= \Bigl|\tau_h(u_{k,i}^h)-\tau(u_{k,i}^h)+\int_{u_{k,i}^h}^{v_{k,i}^h}\Bigl(g([s]_h,x_h([s]_h),\tau_h([s]_h)-g(s,x(s),\tau(s)))\Bigr)ds
 \Bigr|\\
 &\leq |\tau_h(u_{k,i}^h)-\tau(u_{k,i}^h)|+\int_{u_{k,i}^h}^{v_{k,i}^h} L_2\Bigl(h+|x_h([s]_h)-x([s]_h)|+|x([s]_h)-x(s)|\nl
 &\qquad +|\tau_h([s]_h)-\tau([s]_h)|+|\tau([s]_h)-\tau(s)|\Bigr)ds.
\end{align*}
Let $i_{k,t}$ be the largest index such that $u_{k,i_{k,t}}<t$, and $\bar i_{k,t}$ be the largest index so that
$u_{k,\bar i_{k,t}+1}-\delta<t$.
Since $x$ and $\tau$ are Lipschitz continuous on $[u_{k,i},u_{k,i}+\delta]$, we get
for $i=0,\ldots,i_{k,t}$ that
\begin{align*}
 \eta_h(v_{k,i}^h)
 &\leq \eta_h(u_{k,i}^h)+ L_2(1+M_2^*+M_3^*)\delta h +\int_{u_{k,i}^h}^{\min\{u_{k,i}^h+\delta,t\}} L_2\omega_h([s]_h)\,ds.
\end{align*}
Hence, combining it with \eref{e867a} yields
\begin{equation}\label{e867}
m\Bigl(A_{k,i}^h\cap [t_k,t]\Bigr)\leq \frac{1}{\epsilon_0} \eta_h(u_{k,i}^h)+ \Bigl(K_0+1\Bigr)h +\frac{L_2}{\epsilon_0}\int_{u_{k,i}^h}^{\min\{u_{k,i}^h+\delta,t\}} \omega_h([s]_h)\,ds,\qquad i=0,\ldots,i_{k,t}.
\end{equation}
%where $K_0=\frac{1+(1+M_2^*+M_3^*)\delta}{\epsilon_0}$.

In a similar way  we can show that
\begin{equation}\label{e868}
m\Bigl(C_{k,i}^h\cap [t_k,t]\Bigr)\leq \frac{1}{\epsilon_0} \eta_h(u_{k,i+1}^h-\delta)+ K_0h +\frac{L_2}{\epsilon_0}\int_{u_{k,i+1}^h-\delta}^{\min\{u_{k,i+1}^h,t\}} \omega_h([s]_h)\,ds,\qquad i=0,\ldots,\bar i_{k,t}.
\end{equation}
Now we are ready to estimate the integral with the delayed terms.
On the set $B_{k,i}^h$ we can use Lipschitz
continuity of $x$ with Lipschitz constant $M_2^*$ and estimate \eref{e863}, and on the sets $A_{k,i}^h$
and $C_{k,i}^h$ we use estimates \eref{e867} and \eref{e868}, hence
for
$t\in\calI_k\cap[0,\beta]$ we get
\begin{align}
\displaystyle
 \int_{t_k}^{t} &|x([s]_h-[\tau_h([s]_h)]_h)-x(s-\tau(s))|\,ds\nl
 &=\sum_{i=0}^{i_{k,t}-1} \! \int_{A_{k,i}^h}\! |x([s]_h-[\tau_h([s]_h)]_h)-x(s-\tau(s))|\,ds\nl
 &\quad  +\int_{A_{k,i_{k,t}}^h\cap[u_{k,i_{k,t}}^h,t)}\! |x([s]_h-[\tau_h([s]_h)]_h)-x(s-\tau(s))|\,ds\nl
 &\quad+\sum_{i=0}^{i_{k,t}-1} \! \int_{B_{k,i}^h}\! |x([s]_h-[\tau_h([s]_h)]_h)-x(s-\tau(s))|\,ds\nl
 &\quad  +\int_{B_{k,i_{k,t}}^h\cap[u_{k,i_{k,t}}^h,t)}\! |x([s]_h-[\tau_h([s]_h)]_h)-x(s-\tau(s))|\,ds\nl
 &\quad+\sum_{i=0}^{i_{k,t}-1}  \int_{C_{k,i}^h}\! |x([s]_h-[\tau_h([s]_h)]_h)-x(s-\tau(s))|\,ds\nl
 &\quad  +\int_{C_{k,i_{k,t}}^h\cap[u_{k,i_{k,t}}^h,t)}\! |x([s]_h-[\tau_h([s]_h)]_h)-x(s-\tau(s))|\,ds\nl
 &\leq 2M_1^*\sum_{i=0}^{i_{k,t}-1}m(A_{k,i}^h)+ 2M_1^*m\Bigl(A_{k,i_{k,t}}^h\cap[u_{k,i_{k,t}},t)\Bigr)
 + \sum_{i=0}^{i_{k,t}-1} M_2^*\int_{B_{k,i}^h}(Mh+\eta_h(s))\,ds\nl
 &\quad+M_2^*\int_{B_{k,i_{k,t}}^h\cap[u_{k,i_{k,t}},t)}(Mh+\eta_h(s))\,ds
 +2M_1^*\sum_{i=0}^{i_{k,t}-1}m(C_{k,i}^h)+2M_1^*m\Bigl(C_{k,i_{k,t}}^h\cap[u_{k,i_{k,t}},t)\Bigr)\nl
 &\leq K_4\sum_{i=0}^{i_{k,t}} \eta_h(u_{k,i}^h)
 +K_4\sum_{i=0}^{\bar i_{k,t}}\eta_h(u_{k,i+1}^h-\delta)
 +K_1h+\frac{L_2}{\epsilon_0}\int_{t_k}^t \omega_h([s]_h)\,ds+M_2^*\int_{t_k}^t \eta_h(s)\,ds.\label{e869}
\end{align}
Since $K_1\geq 2M_1^*$, \eref{e869} holds for $t\in(t_k,t_k+h)$ too.

We introduce the simplifying notations $w_{k,2i}^h=u_{k,i}^h$ and $w_{k,2i+1}^h=u_{k,i+1}^h-\delta$ for
$i=0,\ldots,i_k$, $w_{k,2i_k+2}=u_{i_k+1}=t_{k+1}$, and let $i_{k,t}^*$ be the largest index such that $w_{k,i_{k,t}^*}^h<t$.
Then
$$
t_k=w_{k,0}^h<w_{k,1}^h<\cdots < w_{k,i_{k,t}^*}^h<t<t_{k+1}.
$$
We comment that $z_h$, $\eta_h$ and $\omega_h$ are continuous at the points $w_{k,1}^h,\cdots, w_{k,i_{k,t}^*}^h$.

Combining \eref{e210a}, \eref{e860} and \eref{e869},  we get for $t\in\calI_k\cap[0,\beta]$ and $k=0,1,\ldots,k_0$
\begin{align}
  z_h(t) &\leq z_h(t_k)+\mu_{1,h}
  +K_2h+K_4\sum_{i=0}^{i_{k,t}^*} \eta_h(w_{k,i}^h)\nl
  &\quad +\int_{t_k}^{t} L_1\Bigl(z_h([s]_h)+z_h([s]_h-[\tau_h([s]_h)]_h)+\frac{L_2}{\epsilon_0}\omega_h([s]_h)+M_2^*\eta_h(s)\Bigr)\,ds.
  \label{e210}
\end{align}

Similarly, taking the difference of \eref{e205} and \eref{e205b}, and using assumption (A2)(iii), and the definition of $\mu_{2,h}$,
we get for $t\in\calI_k$, $k=0,\ldots,k_0$
\begin{align*}
  |\tau_h(t)-\tau(t)|
  &\leq |\tau_h(t_k)-\tau(t_k)|
  +\int_{t_k}^t |g([s]_h,x_h([s]_h),\tau_h([s]_h))-g(s,x(s),\tau(s))|\,ds\\
  &\leq |\tau_h(t_k)-\tau(t_k)|
  +\int_{t_k}^t |g([s]_h,x(s),\tau(s))-g(s,x(s),\tau(s))|\,ds\\
  &\quad +\int_{t_k}^t |g([s]_h,x_h([s]_h),\tau_h([s]_h))-g([s]_h,x(s),\tau(s))|\,ds\\
  &\leq |\tau_h(t_k)-\tau(t_k)|+\mu_{2,h}
   +\int_{t_k}^t L_2\Bigl(|x_h([s]_h)-x([s]_h)|+|x([s]_h)-x(s)|\\
   &\qquad+|\tau_h([s]_h))-\tau([s]_h)|+|\tau([s]_h))-\tau(s)|\Bigr)\,ds.
\end{align*}
Hence, using the definitions of $z_h$ and $\eta_h$ and relations \eref{e216}, \eref{e415},
\eref{e415b} and \eref{e860}, we get
for $t\in\calI_k$, $k=0,\ldots,k_0$
\begin{align}
  \eta_h(t)
  &\leq \eta_h(t_k)+\mu_{2,h}+K_3h
   +\int_{t_k}^t L_2\Bigl(z_h([s]_h)+\eta_h([s]_h)\Bigr)\,ds, \quad t\in\calI_k,\ k=0,\ldots,k_0.\label{e211}
\end{align}
Adding \eref{e210} and \eref{e211} and using the definitions of $\omega_h$, $D_h$ and $L_0$, we get for $t\in\calI_k\cap[0,\beta]$ and $k=0,\ldots,k_0$
\begin{align}
 \omega_h(t)
 &\leq \omega_h(t_k)+D_h + K_4\sum_{i=0}^{i_{k,t}^*} \eta_h(w_{k,i}^h)
 +\int_{t_k}^t L_0\Bigl(\omega_h([s]_h)+\omega_h(s)+\omega_h([s]_h-[\tau_h([s]_h)]_h)\Bigr)\,ds\nl
 &\leq \omega_h(t_k)+D_h +K_4\sum_{i=0}^{i_{k,t}^*} \omega_h(w_{k,i}^h)
 +\int_{t_k}^t 3L_0\sup_{-M_1^*\leq \zeta\leq s}\omega_h(\zeta)\,ds.
 \label{e213}
\end{align}

The continuity of $\tau_h$, $\tau$ and $\eta_h$, the impulsive conditions \eref{e1c} and \eref{ae1c} and assumption (A3) yield
\begin{align}
|\Delta\omega_h(t_k)|&= |\Delta z_h(t_k)|\nl
&= \Bigl| |x_h(t_k)-x(t_k)|-|x_h(t_k-)-x(t_k-)|\Bigr|\nl
&\leq |\Delta x_h(t_k)-\Delta x(t_k)|\nl
&=|I_k(x_h(t_k-))-I_k(x(t_k-))|\nl
&\leq L_3|x_h(t_k-)-x(t_k-)|\nl
&\leq L_3\omega_h(t_k-).\label{e874}
\end{align}
The initial conditions \eref{e2a} and \eref{ae2a} give
$$
\omega_h(t)=|x_h(t)-x(t)|+|\tau_h(t)-\tau(t)|=0,\qquad t\in(-\infty,0].
$$
Suppose $\beta\in\calI_{k^*}$, and $i^*$ be the largest index such that $w_{k^*,i^*}^h<\beta$.
We consider the consecutive intervals
$$
[t_0,w_{0,1}^h),[w_{0,1}^h,w_{0,2}^h),\ldots,[w_{0,2i_0+1}^h,t_1),[t_1,w_{1,1}^h),\ldots,[w_{1,2i_1+1}^h,t_2),
\ldots,[t_{k^*},w_{k^*,1}^h),\ldots,[w_{k^*,i^*}^h,\beta].
$$
We apply inequality \eref{e213}  successively on the above intervals. Since $\omega_h(t_0)=0$,  on the interval
$[t_0,w_{0,1}^h)$ inequality \eref{e213} reduces to
\begin{align}
 \omega_h(t)
 &\leq \omega_h(t_0)+D_h +\int_{t_0}^t 3L_0\sup_{-M_1\leq \zeta\leq s}\omega_h(\zeta)\,ds,\qquad t\in[t_0,w_{0,1}^h).
 \label{e870}
\end{align}
Then inequality \eref{e870} and Lemma~\ref{l00} with $r=M_1^*$, $a_0=0$, $a_1=D_h$, $a_2=0$, $b=3L_0$,
 $c=L_3$, $K=0$ and $T=w_{0,1}^h$, and $w_{k,j}^h\leq\alpha$ and the continuity of $\omega_h$ at $w_{0,1}^h$ yield
\begin{equation}\label{e871}
 \omega_h(t)\leq D_h e^{3L_0 \alpha},\qquad t\in[0,w_{0,1}^h],\quad h\in(0,h_2).
\end{equation}
Suppose
$$
 \omega_h(t)\leq D_h E_k2^{j-1}K_4^{j-1}\Bigl(\sum_{i=0}^k(1+L_3)^i\Bigr)^{j-1}(e^{3L_0 \alpha})^{j},\qquad
 t\in[0,w_{k,j}^h),\quad h\in(0,h_2),
$$
where %$E_0=1$, $m_0=1$,  and
\begin{align*}
 m_0&=0,\qquad m_k= 2i_0+\cdots+2i_{k-1},\qquad k=1,\ldots,k_0,\\
% m_1^*&=0,\qquad m_k^*= 2i_1+\cdots+2i_{k-1}+k-1, \qquad k=2,\ldots,k_0,\\
 E_0&=1,\qquad E_k=2^{m_k+k}K_4^{m_k+k}\Bigl(\sum_{i=0}^{k}(1+L_3)^i\Bigr)^{m_{k}-2i_0+2k}(e^{3L_0 \alpha})^{m_k+2k},
\qquad k=1,\ldots,k_0,
\end{align*}
and consider the next interval $[w_{k,j}^h,w_{k,j+1}^h)$. Then \eref{e213}, the continuity of $\omega_h$ at $w_{k,i}^h$ for $i<2i_k+2$ and $K_4\geq1$ imply the
following very rough estimate
\begin{align}
 \omega_h(t)
 &\leq \omega_h(t_k)+D_h +K_4\sum_{i=0}^{j}D_h E_k2^{i-1}K_4^{i-1}\Bigl(\sum_{i=0}^k(1+L_3)^i\Bigr)^{i-1}(e^{3L_0 \alpha})^{i}
 +\int_{t_k}^t 3L_0\sup_{-M_1^*\leq \zeta\leq s}\omega_h(\zeta)\,ds\nl
 &\leq \omega_h(t_k)+D_hE_k\Bigl(1 +\sum_{i=0}^{j} 2^{i-1}\Bigl)K_4^{j}\Bigl(\sum_{i=0}^k(1+L_3)^i\Bigr)^{j-1}(e^{3L_0 \alpha})^{j}
 +\int_{t_k}^t 3L_0\sup_{-M_1^*\leq \zeta\leq s}\omega_h(\zeta)\,ds\nl
 &\leq \omega_h(t_k)+D_hE_k2^{j}K_4^{j}\Bigl(\sum_{i=0}^k(1+L_3)^i\Bigr)^{j-1}(e^{3L_0 \alpha})^{j}
 +\int_{t_k}^t 3L_0\sup_{-M_1^*\leq \zeta\leq s}\omega_h(\zeta)\,ds,\qquad t\in[t_k,w_{k,j+1}^h).
 \label{e872}
\end{align}
Then inequality \eref{e872} and Lemma~\ref{l00} with $r=M_1^*$, $a_0=0$,
$a_1=D_hE_k2^{j}K_4^{j}\Bigl(\sum_{i=0}^k(1+L_3)^i\Bigr)^{j-1}(e^{3L_0 \alpha})^{j}$, $a_2=0$, $b=3L_0$, $c=L_3$, $K=k$ and $T=v_{k,j+1}^h$ yield
\begin{equation}\label{e208c}
 \omega_h(t)\leq D_hE_k2^{j}K_4^{j}\Bigl(\sum_{i=0}^k(1+L_3)^i\Bigr)^{j}(e^{3L_0 \alpha})^{j+1},\qquad
 t\in[0,v_{k,j+1}^h),\quad h\in(0,h_2).
\end{equation}
Then, applying \eref{e874}, $1+L_3\leq\sum_{i=0}^k(1+L_3)^i$, \eref{e208c} and the definition of $E_{k+1}$, we get
$$
\omega_h(t)\leq D_hE_{k+1},\qquad t\in[0,t_{k+1}],\quad h\in(0,h_2),\quad k=0,\ldots,k^*-1,
$$
and \eref{e208c} holds with $k=k^*$ and $j=j^*$. But then
$$
\omega_h(t)\leq D_hE_{k^*}2^{j^*}K_4^{j^*}\Bigl(\sum_{i=0}^{k^*}(1+L_3)^i\Bigr)^{j^*}(e^{3L_0 \alpha})^{j^*+1}\leq D_hE^*,\qquad t\in[0,\beta],\quad h\in(0,h_2),
$$
which contradicts to \eref{e873} and \eref{e865}. Therefore $\beta=\alpha$, and
\begin{equation}\label{e208b}
\omega_h(t)\leq D_h E^*,\qquad t\in[0,\alpha],\quad h\in(0,h_2)
\end{equation}
holds.
The definition of $D_h$ and relations \eref{e212} and  \eref{e208b}  prove the limit relations  \eref{e221}.
\end{proof}
\bigskip

The definition of $D_h$ and \eref{e208b} immediately have the following  corollary.
\medskip

\begin{corollary}\label{cor1}
 Assume that $T>0$ is finite, $\lambda>0$, (A1)--(A4) hold,
 let $\alpha>0$ and $h^*>0$ be defined by Lemma~\ref{l9}.
 Let $(x,\tau)$ be any solution of the \IIVP{e1}{e2} on $(-\infty,\alpha]$,
where the corresponding delayed time function
$t-\tau(t)$ is piecewise strictly monotone  on $[0,\alpha]$, and suppose the impulsive time moments $t_k$ ($k=-\ell_0,\ldots,k_0$) are not local extreme values or extreme points of the delayed time function, and $\dot\tau(0)\neq1$ and $\dot\tau(\alpha-)\neq1$.
Let $(x_h,\tau_h)$ be the solution of the \IIVP{ae1}{ae2} on $(-\infty,\alpha]$ for $0<h<h^*$.
 Then there exist real constants $N>0$ and $h_1\in(0,h^*)$ such that
 \begin{equation}\label{e221b}
  |x_h(t)-x(t)|\leq Nh\qquad\mbox{and}\qquad |\tau_h(t)-\tau(t)|\leq Nh,\qquad t\in[0,\alpha],\quad h\in(0,h_1).
 \end{equation}
\end{corollary}
\bigskip

Next we formulate the following version of the existence and uniqueness result for the \IIVP{e1}{e2}
for a case when the delayed time function is piecewise strictly monotone.
\smallskip

\begin{theorem}\label{t3}
Assume (A1)--(A4),  let $\lambda>0$,
 let  $\alpha>0$ and $h^*>0$ be defined by Lemma~\ref{l9}, and $(x_h,\tau_h)$ be
the solution of the \IIVP{ae1}{ae2}  on $(-\infty,\alpha]$ for $0<h<h^*$.
Suppose further that there exists a positive sequence $h_i$ tending to 0 such that
 $$
 \lim_{i\to\infty}|x_{h_i}-x|_\PCb=0\qquad\mbox{and}\qquad  \lim_{i\to\infty}|\tau_{h_i}-\tau|_{C([0,\alpha],\setR)}=0
 $$
 for some functions $x\in\PCb$ and $\tau\in C([0,\alpha],\setR)$. Suppose that the corresponding delayed time function $t-\tau(t)$
 is piecewise strictly monotone on $[0,\alpha]$, and  the impulsive time moments $t_k$ ($k=-\ell_0,\ldots,k_0$) are not local extreme values or extreme points of the delayed time function, and $\dot\tau(0)\neq1$ and $\dot\tau(\alpha-)\neq1$. Then
 $(x,\tau)$ is the unique solution of the \IIVP{e1}{e2}  on $[0,\alpha]$ which
satisfies \eref{e314a}--\eref{e315b} with some nonnegative constants $M_1$, $M_2$ and $M_3$.
\end{theorem}
\begin{proof}
The proof  that $(x,\tau)$ is a solution of the \IIVP{e1}{e2} which
satisfies \eref{e314a}--\eref{e315b} follows the same lines as the proof of Theorem~\ref{t1}.
The key observation is that if the delayed time function $t-\tau(t)$ is piecewise strictly monotone, then the set $\calU$ defiend by
\eref{e855} has Lebesgue measure 0, consequently,  relation \eref{e853} holds. The rest of the argument is identical to
that used in the proof of Theorem~\ref{t1}.

For the proof of the uniqueness, assume that $(\bar x,\bar\tau)$ is also a solution of the \IIVP{e1}{e2}.
Theorem~\ref{t2} yields that
 $$
 \lim_{i\to\infty}|x_{h_i}-\bar x|_\PC=0\qquad\mbox{and}\qquad  \lim_{i\to\infty}|\tau_{h_i}-\bar\tau|_{C([0,\alpha],\setR)}=0,
 $$
which implies that $(x,\tau)=(\bar x,\bar\tau)$, i.e., $(x,\tau)$ is the unique solution of the \IIVP{e1}{e2}.
\end{proof}
\medskip

\section{Numerical examples}\label{s5}
\setcounter{equation}0
\setcounter{theorem}0

In this section we present two examples to illustrate the convergence of the numerical
scheme \eref{ae1}--\eref{ae2}.
\medskip

\begin{example}\label{example1}\rm
 Consider the initial value problem (IVP)
 \begin{align}
  \dot x(t) &= -\Bigl(x(t-\tau(t))-1\Bigr) \exp\Bigl(-0.3(x(t)-1) \sin 5t-2\Bigr),\qquad t\in[0,4],\label{ex10}\\
  \dot \tau(t) &=2-\tau(t)+1.5(x(t)-1)\cos 5t,\qquad t\in[0,4],\label{ex11}\\
   x(t) &=
   \left\{\begin{array}{ll}
   e^{-t}+1,\quad &t\in(-3,0],\\
   e^{3}+1,\quad &t\in(-\infty,-3],
   \end{array}\right.\label{ex12}\\
  \tau(t) &=2,\qquad t\in(-\infty,0].\label{ex13}
 \end{align}
 Here we do not have an impulsive condition associated to the IVP.

 Clearly, conditions (A1) (i)--(ii), (A2) (i), (iii), (A3) and (A4) (i), (ii) are satisfied. For condition (A2) (ii)
 consider
 $$
 2+1.5(u-1)\cos 5t.
 $$
 Unfortunately, it is not positive for all $(t,u)\in[0,4]\times \setR$. But define the function
 $$
 g(t,u,w)=
 \left\{\begin{array}{ll}
          2-w-1.5\cos 5t,\qquad u<0,\ w\in\setR,\\
          2-w+1.5(u-1)\cos 5t,\qquad u\in[0,2.3],\ w\in\setR,\\
          2-w+1.95\cos 5t,\qquad u>2.3,\ w\in\setR,\\
        \end{array}\right.
 $$
 Note that $0<\phi(0)<2.3$.
 If we replace \eref{ex11} with $\dot \tau(t)=g(t,x(t),\tau(t))$, then all conditions (A1)--(A4) are satisfied,
 so the corresponding IVP has a unique solution, and until $x(t)\in[0,2.3]$, the solution coincides with
 that of the \IVP{ex10}{ex13}.

 It can be checked that the exact
 solution of the \IVP{ex10}{ex13} is
 \begin{align*}
  x(t) &= e^{-t}+1, \qquad \tau(t)=0.3e^{-t}\sin 5t+2.
 \end{align*}
 Note that the delayed argument function $t-\tau(t)$ is not monotone increasing in this IVP.
 We use our numerical approximation scheme to approximate the solution.
 In the next table we show the numerical values and the corresponding errors for $h=0.1$, $0.01$ and $0.001$.
 We observe convergence of the approximate solution to the exact solution, and it looks that the order of convergence is linear.
 In Figure~1 we plotted the exact and the approximate solution corresponding to $h=0.05$.

 \begin{center}
 \small
 \begin{tabular}{|c|ccccc|}
 \hline
 & \multicolumn{5}{c|}{$h=0.1$}\\
$s_i$  &  $i$   &    $x_h(s_i)$ & $\tau_h(s_i)$ & $e_x(s_i)$  & $e_\tau(s_i)$\\
\hline
1.0 &  10 &  1.32583334  & 1.91672237 &  4.205e-02  &  2.255e-02 \\
2.0 &  20 &  1.08942499  & 2.00168457 &  4.591e-02  &  2.377e-02 \\
3.0 &  30 &  1.00646891  & 2.00745845 &  4.332e-02  &  2.254e-03 \\
4.0 &  40 &  0.97979564  & 1.99518831 &  3.852e-02  &  9.828e-03 \\
\hline
 \hline
 & \multicolumn{5}{c|}{$h=0.01$}\\
$s_i$  &  $i$   &    $x_h(s_i)$ & $\tau_h(s_i)$ & $e_x(s_i)$  & $e_\tau(s_i)$\\
\hline
1.0 & 100 &  1.36446215  & 1.89606294 &  3.417e-03  &  1.893e-03 \\
2.0 & 200 &  1.13159363  & 1.98032187 &  3.742e-03  &  2.409e-03 \\
3.0 & 300 &  1.04629895  & 2.00991036 &  3.488e-03  &  1.976e-04 \\
4.0 & 400 &  1.01527840  & 2.00427610 &  3.037e-03  &  7.403e-04 \\
\hline
 \hline
 & \multicolumn{5}{c|}{$h=0.001$}\\
$s_i$  &  $i$   &    $x_h(s_i)$ & $\tau_h(s_i)$ & $e_x(s_i)$  & $e_\tau(s_i)$\\
\hline
1.0 & 1000 &  1.36753037  & 1.89435808 &  3.491e-04  &  1.886e-04 \\
2.0 & 2000 &  1.13496055  & 1.97815354 &  3.747e-04  &  2.411e-04 \\
3.0 & 3000 &  1.04943371  & 2.00973378 &  3.534e-04  &  2.101e-05 \\
4.0 & 4000 &  1.01801021  & 2.00494219 &  3.054e-04  &  7.416e-05 \\
\hline
 \end{tabular}\\
Table 1: Approximate solution of the \IVP{ex10}{ex13}. $s_i=ih$, $e_x(s_i)=|x_h(s_i)-x(s_i)|$ and $e_\tau(s_i)=|\tau_h(s_i)-\tau(s_i)|$
\end{center}

\begin{center}
\begin{minipage}{0.49\textwidth}
\begin{center}
 {\includegraphics[width=0.85\textwidth]{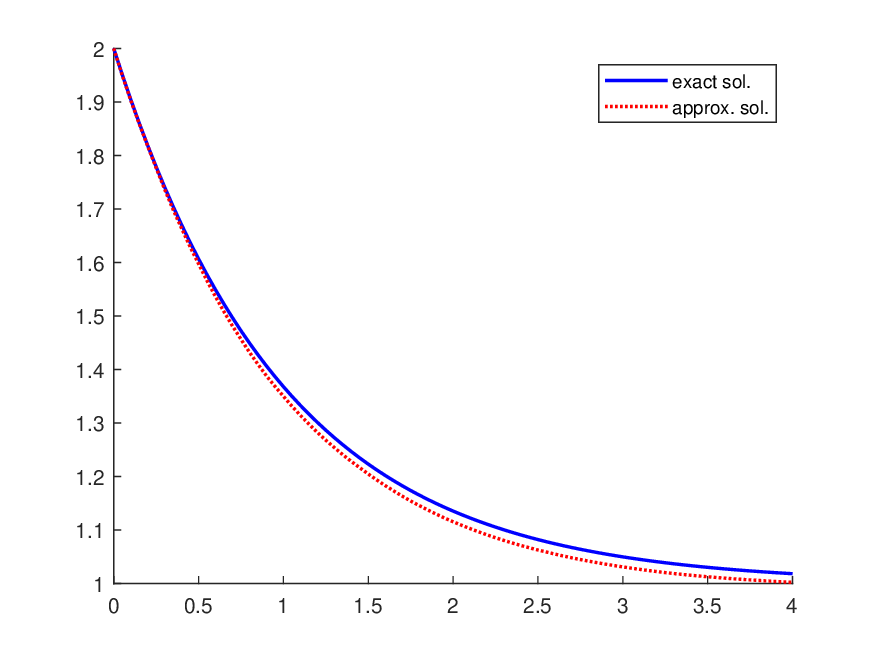}}\\[-0.3cm]
$x_h(t)$
\end{center}
\end{minipage}
 \
 \begin{minipage}{0.49\textwidth}
 \begin{center}
 {\includegraphics[width=0.85\textwidth]{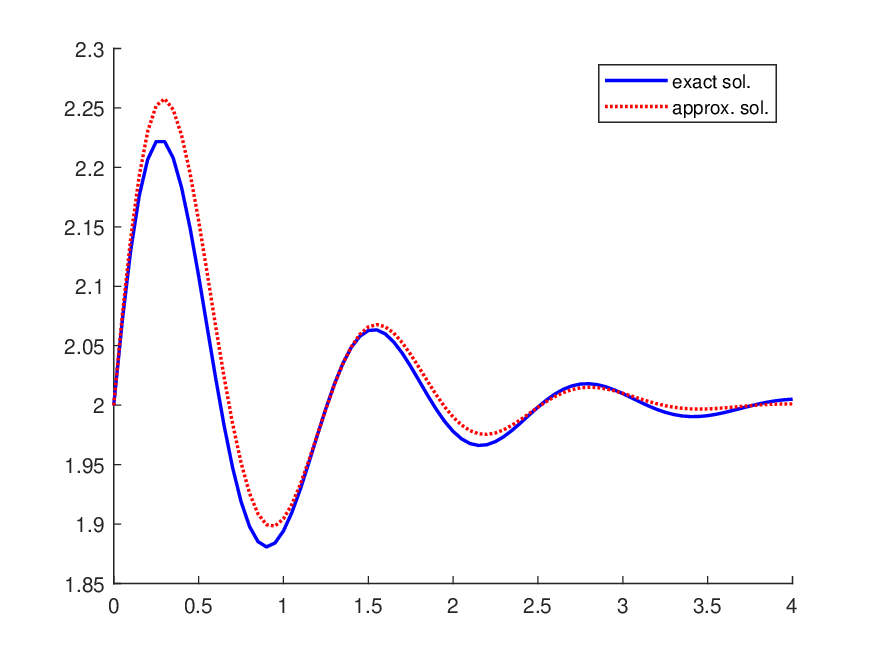}}\\[-0.3cm]
$\tau_h(t)$
\end{center}
\end{minipage}\\
Figure 1: Exact and approximate solution of the \IVP{ex10}{ex13}, $h=0.05$
\end{center}
\end{example}

\begin{example}\rm
 Consider the IIVP
 \begin{align}
  \dot x(t) &= -\frac 15x(t-\tau(t)),\qquad t\in[0,5],\label{ex1}\\
  \dot\tau(t) &= 1+\frac{1}{4}x(t)-\frac 1{2}\tau(t),\qquad t\in[0,5],\label{ex2}\\
  \Delta x\left(\frac 34\right) &= 2,\qquad \Delta x\left(\frac32\right)=-\frac34,\label{ex3}\\
  x(t) &= 1,\qquad t\in[-3,0],\label{ex4}\\
  \tau(t) &= 2,\qquad t\in[-3,0].\label{ex5}
 \end{align}
 In the region $u>-4$ condition $g(t,u,0)>0$ is satisfied, so similarly to the argument used in Example~\ref{example1},
 Theorem~\ref{t1} yields that the \IIVP{ex1}{ex5} has a unique solution on $[0,\alpha]$ for some $\alpha>0$.

 Initially $t-\tau(t)$ is negative, so $x(t-\tau(t))=1$. Hence integrating \eref{ex1} and using $x(0)=1$ we get $x(t)=-\frac15t+1$.
 If we substitute this formula to \eref{ex2}, we can solve the corresponding ODE with the initial condition $\tau(0)=2$, and we
 get $\tau(t)=\frac{27}{10}-\frac 1{10}t-\frac7{10}e^{-\frac 12t}$. We can check that $t-\tau(t)<0$ for $t\in[0,3/4]$, hence
 this is the solution of the IIVP on  $[0,3/4]$.

 Using the first impulsive condition of \eref{ex3} we get $x(t)=-\frac15t+3$ for $t\geq 3/4$ but close to $3/4$.
 Then we can solve
 \eref{ex2} using the initial condition $\tau(3/4)=\frac{21}8-\frac 7{10}e^{-3/8}$, hence we get
 $\tau(t)=\frac{37}{10}-\frac1{10}t-\frac1{10}e^{-\frac12t}\cdot (10e^{3/8}+7)$. We have $t-\tau(t)<0$ for $t\in[3/4,3/2]$, so this is
 the solution on this interval.

 At $t=\frac 32$ we use the second impulsive condition of \eref{ex3}, so we have $x(t)=-\frac15t+\frac94$
 for $t\geq 3/2$ close to $3/2$.
 Then solving the corresponding  \eref{ex2} with the initial condition $\tau(3/2)=\frac{71}{20}-e^{-3/8}-\frac 7{10}e^{-3/4}$
 we get $\tau(t)=\frac{133}{40}-\frac1{10}t-\frac1{40}e^{-\frac12t}(-15e^{3/4}+40e^{3/8}+28)$. We can check that
 $t-\tau(t)<0$ for $t\in[3/2,T_1)$ and $T_1-\tau(T_1)=0$ with $T_1= 2.702320978$. Therefore we get
 $(x(T_1),\tau(T_1))=(2.459535804, 2.871755310)$.

 Hence for $t>T_1$ but close to $T_1$ we have $x(t)=-\frac15t+1$. Using this relation \eref{ex1} and \eref{ex2} are simplified to
 \begin{align*}
  \dot x(t) &= \frac{1}{25}-\frac{1}{25}\tau(t)-\frac15,\\
  \dot\tau(t) &=1+\frac{1}{4}x(t)-\frac 15\tau(t).
 \end{align*}
 We solved this system of ODEs with the initial condition $(x(T_1),\tau(T_1))=(2.459535804, 2.871755310)$ using Maple to get the
 formula of the solution on the next interval: $x(t)=-110.0000000-0.1490476191\cdot \exp(-0.4791287848t)+112.5160738\cdot \exp(-0.02087121525t)+2.000000001t$ and $\tau(t)=-55.00000002+58.70867993\cdot \exp(-0.02087121525t)-1.785325143\cdot \exp(-0.4791287848t)+1.000000001t$. We get that $0\leq t-\tau(t)<\frac 34$ for $t\in[T_1,T_2)$ and $T_2-\tau(T_2)=\frac 34$
 for $T_2=3.488254843$.
 Then we have the solution of the \IIVP{ex1}{ex5} on $[T_1,T_2]$. In a similar way we can solve the IVP on a next interval.
 We can obtain that the exact solution of the \IIVP{ex1}{ex5} has the formula
 $$
 x(t)=\left\{
 \begin{array}{l}
  -\frac15t+1,\qquad t\in[0,0.75),\\
  -\frac15t+3,\qquad t\in[0.75,1.5),\\
  -\frac15t+\frac94,\qquad t\in[1.5,2.702320978),\\
  -110.0000000-0.1490476191e^{-0.4791287848t}+112.5160738e^{-0.02087121525t}\\
  \qquad +2.000000001t,\qquad   t\in[2.702320978,3.488254843),\\
  -130.0000000-0.3512993334e^{-0.4791287848t}+134.0673650e^{-0.02087121525t}\\
  \qquad +2.000000001t,\qquad t\in[3.488254843,4.218920247),\\
  -122.5000000-0.2436621429e^{-0.4791287848t}+125.8614408e^{-0.02087121525t}\\
  \qquad +2.000000001t,\qquad t\in[4.218920247,5],
 \end{array}\right.
 $$
and
 $$
 \tau(t)=\left\{
 \begin{array}{l}
  \frac{27}{10}-\frac 1{10}t-\frac7{10}e^{-\frac 12t},\qquad t\in[0,0.75),\\
  \frac{37}{10}-\frac1{10}t-\frac1{10}e^{-\frac12t}\cdot (10e^{3/8}+7),\qquad t\in[0.75,1.5),\\
  \frac{133}{40}-\frac1{10}t-\frac1{40}e^{-\frac12t}(-15e^{3/4}+40e^{3/8}+28),\qquad t\in[1.5,2.702320978),\\
  -55.00000002+58.70867993 e^{-0.02087121525t}-1.785325143 e^{-0.4791287848t}\\
  \qquad +1.000000001t,\qquad   t\in[2.702320978,3.488254843),\\
  -65.00000002+69.95372081 e^{-0.02087121525t}-4.207940644 e^{-0.4791287848t}\\
  \qquad +1.000000001t,\qquad t\in[3.488254843,4.218920247),\\
  -61.25000002+65.67203060 e^{-0.02087121525t}-2.918638192 e^{-0.4791287848t}\\
  \qquad +1.000000001t,\qquad t\in[4.218920247,5].
 \end{array}\right.
 $$
 Table 2 below contains the approximate solutions of the \IIVP{ex1}{ex5} for $h=0.1$, $0.01$ and $0.001$.
 Since the exact solution $x(t)$ is linear on $[0,3/4)$, $[3/4,3/2)$ and on $[3/2,T_1)$, the approximate
 solution \eref{ae1}--\eref{ae2} gives back the exact solution. In the next table we observe
 rounding error only due to the calculations. But for $t>T_1$ and for the $\tau$ component we see that
 if the stepsize $h$ decreases, the error decreases too. Figure 2 shows the numerical solution for $h=0.1$.
 Even for this relatively big stepsize the approximation is  good.

 \begin{center}
 \small
 \begin{tabular}{|c|ccccc|}
 \hline
 & \multicolumn{5}{c|}{$h=0.1$}\\
$s_i$  &  $i$   &    $x_h(s_i)$ & $\tau_h(s_i)$ & $e_x(s_i)$  & $e_\tau(s_i)$\\
\hline
1.0 &  10 &  2.80000000  & 2.27838414 &  0.000e+00  &  1.455e-02 \\
2.0 &  20 &  1.85000000  & 2.63913961 &  4.441e-16  &  1.487e-02 \\
3.0 &  30 &  1.65120000  & 2.73410744 &  4.520e-04  &  1.277e-02 \\
4.0 &  40 &  1.28120000  & 2.74384615 &  3.642e-03  &  1.213e-02 \\
5.0 &  50 &  0.85780000  & 2.65671910 &  9.131e-03  &  8.423e-03 \\
\hline
 \hline
 & \multicolumn{5}{c|}{$h=0.01$}\\
$s_i$  &  $i$   &    $x_h(s_i)$ & $\tau_h(s_i)$ & $e_x(s_i)$  & $e_\tau(s_i)$\\
\hline
1.0 & 100 &  2.80000000  & 2.28930718 &  5.329e-15  &  3.624e-03 \\
2.0 & 200 &  1.85000000  & 2.62435492 &  1.599e-14  &  8.166e-05 \\
3.0 & 300 &  1.65165200  & 2.72176158 &  3.649e-09  &  4.283e-04 \\
4.0 & 400 &  1.27822400  & 2.73245016 &  6.656e-04  &  7.387e-04 \\
5.0 & 500 &  0.86738000  & 2.64907081 &  4.490e-04  &  7.750e-04 \\
\hline
 \hline
 & \multicolumn{5}{c|}{$h=0.001$}\\
$s_i$  &  $i$   &    $x_h(s_i)$ & $\tau_h(s_i)$ & $e_x(s_i)$  & $e_\tau(s_i)$\\
\hline
1.0 & 1000 &  2.80000000  & 2.29257085 &  2.220e-14  &  3.608e-04 \\
2.0 & 2000 &  1.85000000  & 2.62428141 &  4.396e-14  &  8.156e-06 \\
3.0 & 3000 &  1.65165184  & 2.72137602 &  1.636e-07  &  4.270e-05 \\
4.0 & 4000 &  1.27825344  & 2.73185604 &  6.950e-04  &  1.446e-04 \\
5.0 & 5000 &  0.86745468  & 2.64851590 &  5.237e-04  &  2.201e-04 \\
\hline
 \end{tabular}\\
Table 2: Approximate solution of the \IIVP{ex1}{ex5}. $s_i=ih$, $e_x(s_i)=|x_h(s_i)-x(s_i)|$ and $e_\tau(s_i)=|\tau_h(s_i)-\tau(s_i)|$
\end{center}

\begin{center}
\begin{minipage}{0.49\textwidth}
\begin{center}
 {\includegraphics[width=0.85\textwidth]{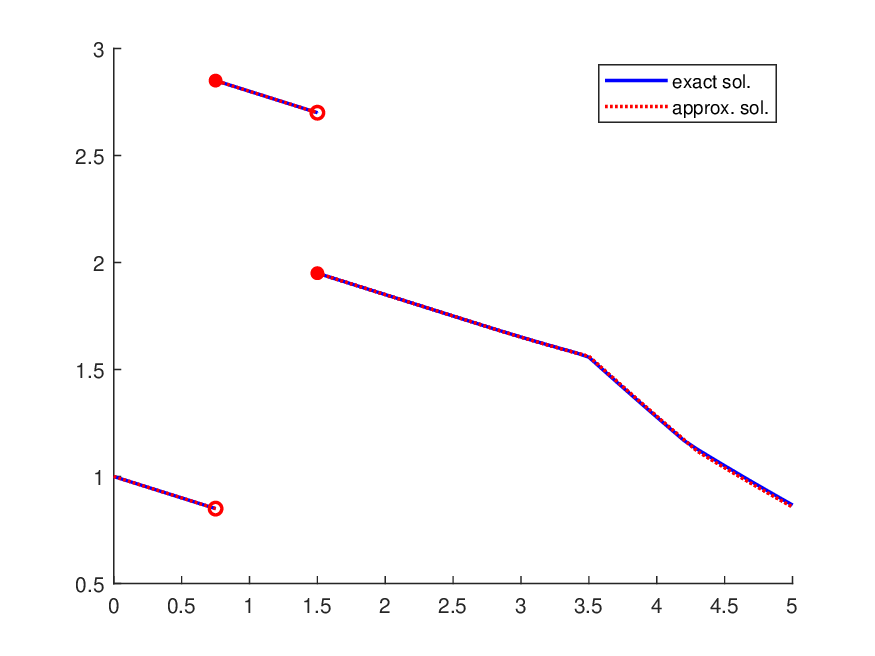}}\\[-0.3cm]
$x_h(t)$
\end{center}
\end{minipage}
 \
 \begin{minipage}{0.49\textwidth}
 \begin{center}
 {\includegraphics[width=0.85\textwidth]{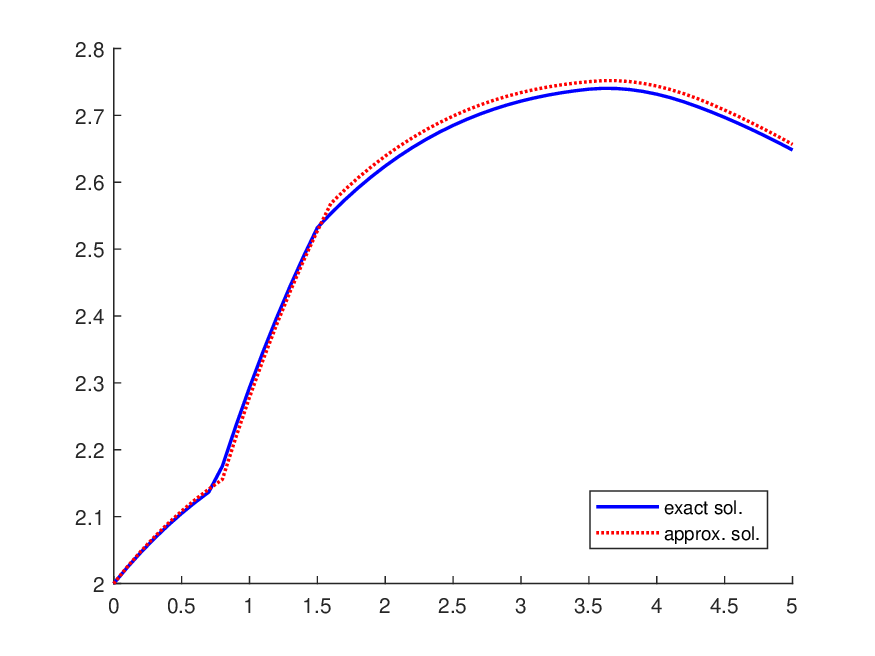}}\\[-0.3cm]
$\tau_h(t)$
\end{center}
\end{minipage}\\
Figure 2: Exact and approximate solution of the \IVP{ex1}{ex5}, $h=0.1$
\end{center}
\end{example}

\section{Conclusion}
\setcounter{equation}0
\setcounter{theorem}0

In this paper we considered a class of SD-DDEs with adaptive delays and impulses.
Using a Picard--Lindel\"of type argument, we showed local existence and uniqueness results
under  the condition that the delayed time function is monotone or piecewise strictly monotone. We introduced a numerical scheme with the help of EPCAs, and
obtained a sequence of approximate solutions which converges to a solution of the
SD-DDE under natural conditions. This simple, first-order approximation scheme could be used
successfully in the proof of the existence of a solution in a more complex problem.
An interesting future research problem is to prove the convergence of the numerical approximation under conditions where a more relaxed condition is used instead of the piecewise strict monotonicity.

Finally, we mention that if we replace the impulsive condition \eref{ae1c} with
$$
 \Delta x_h([t_k]_h) = I_k(x_h([t_k]_h-)),\qquad k=1,2,\ldots,K,
$$
then the impulsive time moments appear at  mesh points of the numerical scheme,
which simplifies the computation of the approximation. On the other hand,
since the jump times of the exact solution and the approximate solution are
different (in general), the proof
of the convergence becomes more complicated. This is an open problem
which can be studied in a future work.

\section*{Acknowledgements}

The author thanks the support  of the Hungarian National Research, Development and Innovation Office grant no. K139346.

\bibliographystyle{plain}
\bibliography{hartung}

\end{document}